\documentclass[12pt]{amsart}
\usepackage{fullpage}
\usepackage{tabu}
\usepackage{amssymb} 
\usepackage{amsmath} 
\usepackage{amscd}
\usepackage{amsbsy}
\usepackage{comment, enumerate}
\usepackage[matrix,arrow]{xy}
\usepackage{mathrsfs}
\usepackage{color}
\usepackage{mathtools,caption}
\usepackage{todonotes}
\usepackage{tikz-cd}

\definecolor{darkgreen}{rgb}{0,0.5,0}

\usepackage[
        colorlinks, citecolor=darkgreen,
        backref,
]{hyperref}
\usepackage{cleveref}

\newcommand{\F}{\mathbb{F}}

\newcommand{\Q}{\mathbb{Q}}

\newcommand{\Z}{\mathbb{Z}}

\newcommand{\kbar}{{\kbar}}

\newcommand{\rhobar}{{\bar{\rho}}}

\newcommand{\p}{\mathfrak{p}}

\newcommand{\calO}{\mathcal{O}}

\newcommand{\Fp}{\mathfrak{p}}
\newcommand{\Fq}{\mathfrak{q}}

\DeclareMathOperator{\Ind}{Ind}

\DeclareMathOperator{\Tr}{Tr}

\newcommand{\GL}{\operatorname{GL}}

\numberwithin{equation}{section}

\newtheorem{theorem}{Theorem}[section]
\newtheorem{lemma}[theorem]{Lemma}

\newtheorem{proposition}[theorem]{Proposition}

\theoremstyle{definition}

\theoremstyle{remark}
\newtheorem{remark}[theorem]{Remark}

\begin{document}

\title[Reduction of Weil--Deligne Representations]{Reduction of Weil--Deligne Representations}

\date{\today}


\author{Imin Chen}

\address{Department of Mathematics, Simon Fraser University\\
Burnaby, BC V5A 1S6, Canada.} 
\email{ichen@sfu.ca}

\author{Deniz Su\"ozer}

\address{Department of Mathematics, Simon Fraser University\\
Burnaby, BC V5A 1S6, Canada.}
\email{deniz\_suozer@sfu.ca}

\begin{abstract}
Let $p$ and $\ell$ be distinct odd primes. For a finite extension $F/\Q_p$, the local Langlands correspondence states that there is a canonical bijection between irreducible, smooth representations of $\GL_n(F)$ and $n$-dimensional, $\Phi$-semisimple Weil--Deligne representations of the Weil group $W_F$. Given two $2$-dimensional, semisimple, continuous representations $\rho_1, \rho_2$ of $W_F$ with images in $\GL_2(\calO_K)$, where $K/\Q_\ell$ is a finite extension with maximal ideal $\lambda \subset \calO_K$, a natural question to ask is when their mod $\lambda$ reductions $\overline{\rho}_1$ and $\overline{\rho}_2$ are isomorphic. In this paper, we give a complete characterization of when the reductions are isomorphic. For this, we first give a description of when the reductions of these representations are decomposable or irreducible, utilizing the correspondence between continuous $\ell$-adic representations of $W_F$ and $\ell$-adic Weil--Deligne representations. We conclude with some examples which arise in the modular method for solving generalized Fermat equations.

\end{abstract}

\thanks{}

\maketitle

{
\hypersetup{linkcolor=black}
\setcounter{tocdepth}{1}
\tableofcontents
}

\section{Introduction}

In this paper, we are interested in the following question: Let $F$ be a finite extension of $\Q_p$ and $K$ be a finite extension of $\Q_\ell$ with maximal ideal $\lambda \subset \calO_K$, where $p$ and $\ell$ are distinct odd primes. Let $\rho_1, \rho_2 : W_F \rightarrow \GL_2(\overline{\Q}_\ell)$ be two semisimple, continuous representations such that $\rho_1(W_F), \rho_2(W_F) \subset \GL_2(\calO_K)$. Under what conditions do we have $\overline{\rho}_1 \cong \overline{\rho}_2$?

The question can also be stated in terms of invariant $\calO_K$-lattices. Let $\rho: W_F \rightarrow \GL_2(\overline{\Q}_\ell)$ be an $\ell$-adic representation such that $\rho(W_F) \subset \GL_2(\calO_K)$. By picking a basis $\beta$ for $V := K^2$, we can identify $\GL_2(K)$ with $\GL(V)$. Let $\Lambda \subset V$ be the $\calO_K$-span of $\beta$. Then $\Lambda$ is an $\calO_K$-lattice in $V$ by definition. Further, $\rho(W_F)\Lambda \subseteq \Lambda$ since $\rho(v)$ is an $\calO_K$-linear combination of $\beta$ for any $v \in \beta$. Thus, $\Lambda$ is a $\rho$-invariant lattice in $V$.

Conversely, suppose that $\rho: W_F \rightarrow \GL(V)$ is a representation of $W_F$ with a $\rho$-invariant $\calO_K$-lattice $\Lambda \subset V$, where $V := K^2$. Then there exists an $\calO_K$-basis $\beta$ for $\Lambda$ that is also a $K$-basis for $V$, and we obtain the representation $[\rho]_\beta: W_F \rightarrow \GL_2(K) \hookrightarrow \GL_2(\overline{\Q}_\ell)$. Since $\Lambda$ is invariant under $\rho$, for any $v \in \beta$ and $g \in W_F$, $\rho(g)v$ is still in $\Lambda$. Consequently, $[\rho]_\beta(g)$ is in $\GL_2(\calO_K)$. Thus, we have $\rho(W_F) \subset \GL_2(\calO_K)$.

So far, we have seen that a representation $\rho$ whose image lies in $\GL_2(\calO_K)$ corresponds to a $\rho$-invariant $\calO_K$-lattice $\Lambda$. To understand the reduction $\overline{\rho}$ in a coordinate-free way, we consider the $\calO_K$-submodule $\lambda\Lambda$ defined by
\begin{equation*}
    \lambda\Lambda := \left\{\sum_{i=1}^n a_iv_i : a_i \in \lambda, v_i \in \Lambda, n \in \mathbb{N}\right\}.
\end{equation*}

Observe that $\lambda\Lambda$ is invariant under $\rho$. This invariance induces an action on the quotient module $\Lambda/\lambda\Lambda$, which is a $2$-dimensional vector space over the finite field $\calO_K/\lambda$. The reduction $\overline{\rho}$ corresponds precisely to the action of $W_F$ on $\Lambda/\lambda\Lambda$ after a choice of basis.

Any basis $\beta$ for $\Lambda$ can be reduced mod $\lambda\Lambda$ to obtain a basis for $\Lambda/\lambda\Lambda$, and a different choice of basis for $\Lambda$ yields an isomorphic reduction $\overline{\rho}$. This is to be expected; a change of basis for the $\calO_K$-lattice $\Lambda$ corresponds to conjugation by some $A \in \GL_2(\calO_K)$, and the representations $\rho$ and $A\rho A^{-1}$ remain isomorphic after reduction.

Returning to the question we have posed at the start of this section, recall the correspondence between continuous $\ell$-adic representations of $W_F$ and $\ell$-adic Weil--Deligne representations. Our strategy for answering the question will be to take an arbitrary pair of semisimple, continuous representations $\rho_1$ and $\rho_2$ of $W_F$; look at the corresponding Weil--Deligne representations; and use the classification of $2$-dimensional, $\Phi$-semisimple Weil--Deligne representations.

As we have mentioned, we assume throughout this paper that $p > 2$, which means that exceptional supercuspidal representations do not exist. This leaves us with principal series, special representations, and nonexceptional supercuspidal representations.

In \Cref{decomposability section}, we will first give the criteria for the reducibility and decomposability of reductions of principal series, nonexceptional supercuspidal representations, and special representations. For this, we will use the decomposability criterion for arbitrary $2$-dimensional representations, stated in \Cref{thm: decomposability criterion for reducible representations}.

In \Cref{comparison section}, we will compare the reductions of pairs of continuous $\ell$-adic representations of $W_F$ for which the corresponding Weil--Deligne representations are principal series, nonexceptional supercuspidal representations, or special representations. It is typically more difficult to compare reductions that are indecomposable. For this, we will use the description of when two $2$-dimensional representations that are reducible and indecomposable are isomorphic, stated in \Cref{thm: isomorphism of indecomposable representations}. We will apply this description to all possible pairings of principal series, special representations, and nonexceptional supercuspidal representations.

One motivation for studying when two $\ell$-adic representations of $W_F$ are isomorphic under reduction comes from the {\it modular method} for resolving Diophantine equations. Let $K$ be a totally real field. In this method, we attach to a putative solution $(a,b,c)$ of a ternary Diophantine equation a Frey abelian variety $A/K$ which gives rise to a $\lambda$-adic representation $\rho_{A,\lambda} : G_K \rightarrow \GL_2(K_\lambda)$ where $\lambda \mid \ell$ is a prime of $K$. The reduction $\rhobar_{A,\lambda}$ is shown to be isomorphic to the reduction $\rhobar_{f,\lambda'}$ of the $\lambda'$-adic representation attached to finitely many Hilbert newforms $f$.

In order to show the non-existence of a putative solution $(a,b,c)$ we would like to show that $\rhobar_{A,\lambda} \not\simeq \rhobar_{f,\lambda'}$ for any $f$ and $\lambda, \lambda'$, that is, we would like to distinguish residual representations. One way to do this is to restrict to a decomposition group at $p$. Regarding $K_\lambda$ and $K_{\lambda'}$ as subfields of $\overline{\Q}_\ell$ and restricting further from a decomposition group to a Weil group, this question then becomes the main question of this paper. A variant for distinguishing residual representations which is often used in the modular method is to ask if we already have $\rhobar_{A,\lambda} \not\simeq \rhobar_{f,\lambda'}$ after restriction to inertia.

\section{Decomposability of Reductions of $\ell$-adic Representations of the Weil Group}
\label{decomposability section}

Let $p$ and $\ell$ be odd primes with $\ell \neq p$. Let $F$ be a finite extension of $\Q_p$ with ring of integers $\calO_F$, maximal ideal $\p$, and residue field $k_F$ of order $q$. Let $E$ be a quadratic extension of $F$, with corresponding ring of integers $\calO_E$ and residue field $k_E$. We denote the absolute Galois group, the Weil group, the inertia group, and the wild inertia group of $F$ by $G_F, W_F, I_F$, and $P_F$, respectively.

Let $K$ be a finite extension of $\Q_\ell$, with ring of integers $\calO_K$, maximal ideal $\lambda$, and residue field $k_K$.

By $\omega$, we denote the $1$-dimensional representation of $W_F$ corresponding to the absolute value on $F^\times$ via local class field theory. Fixing a Frobenius element $\Phi$, $\omega$ is explicitly given by $\omega(\Phi^k\tau) = q^{-k}$ for any $\tau \in I_F, k \in \Z$. We denote by $t_\ell: I_F \rightarrow \Z_\ell$ a fixed continuous, surjective homomorphism.

When considering a supercuspidal representation $\rho$ induced from $W_E$, we use $\sigma \in W_F \setminus W_E$ to denote a fixed representative of the nontrivial coset of $W_E$ in $W_F$.

In this section, we will give the criteria for the reducibility and decomposability of reductions of $2$-dimensional, semisimple, continuous representations of $W_F$. We can identify such representations with $\Phi$-semisimple Weil--Deligne representations of $W_F$. As $p > 2$, there are three types of Weil--Deligne representations: principal series, special representations, and nonexceptional supercuspidal representations. We will analyze each type and describe when the associated continuous representation has a reducible or decomposable reduction.

It is relatively straightforward to show that only representations whose associated Weil--Deligne representation is a nonexceptional supercuspidal representation can have irreducible reductions. Thus, the main goal of this section is to describe when reductions that are reducible are decomposable. For this, we will make use of the following, more general, theorem:

\begin{theorem}
\label{thm: decomposability criterion for reducible representations}
    Let $G$ be a group and $\rho:G \rightarrow \GL_2(L)$ be a representation of $G$ given by
    \begin{equation*}
        \rho(g) = \begin{pmatrix}
            \chi_1(g) & \psi(g) \\
            0 & \chi_2(g)
        \end{pmatrix},
    \end{equation*}
    where $\psi:G \rightarrow L$ is a function and $\chi_1, \chi_2$ are $1$-dimensional representations of $G$ over the field $L$. Then $\rho$ is decomposable if and only if $\psi = c(\chi_2-\chi_1)$ for some $c \in L$.
\end{theorem}

\begin{proof}
    See Suözer \cite[Theorem 2.4.22]{Deniz}.
\end{proof}

We will often refer to \Cref{thm: decomposability criterion for reducible representations} to determine when the reductions of continuous representations of $W_F$ are decomposable.

As mentioned at the start of this section, our strategy will be to use the bijection between semisimple, continuous representations of $W_F$ and $\Phi$-semisimple Weil--Deligne representations, which follows from the following theorem, often called ``Grothendieck's $\ell$-adic monodromy theorem'':

\begin{theorem}
\label{grothendieck's l-adic monodromy thm}
    Let $\rho:W_F \rightarrow \GL_n(\overline{\Q}_\ell)$ be a finite-dimensional continuous representation. There is a unique nilpotent endomorphism $N_\rho \in M_n(\overline{\Q}_\ell)$ such that
    \begin{equation*}
        \rho(\tau) = \exp(t_\ell(\tau)N_\rho)
    \end{equation*}
    for all elements $\tau$ of some open subgroup of $I_F$, where $t_\ell:I_F \rightarrow \Z_\ell$ is a fixed, continuous surjection.
\end{theorem}

\begin{proof}
    See Bushnell and Henniart \cite[32.5 Theorem]{Bushnell-Henniart}.
\end{proof}

\begin{theorem}
\label{Weil--Deligne and continuous weil representation correspondence}
    Let $\operatorname{Rep}_{\overline{\Q}_\ell}(W_F)$ and $\operatorname{WD-Rep}_{\overline{\Q}_\ell}(W_F)$ be the categories of finite-dimensional continuous representations and Weil--Deligne representations of $W_F$ over $\overline{\Q}_\ell$, respectively. Let $\Phi \in W_F$ be a Frobenius element and $t_\ell:I_F \rightarrow \Z_\ell$ be a continuous surjection. Define the representation $\rho_\Phi$ by
    \begin{equation}
    \label{Weil--Deligne correspondence bijection}
        \rho_\Phi(\Phi^k\tau) = \rho(\Phi^k\tau)\exp(-t_\ell(\tau)N_\rho)
    \end{equation}
    for all $k \in \Z$ and $\tau \in I_F$, where $N_\rho$ is defined as in \Cref{grothendieck's l-adic monodromy thm}. The assignment
    \begin{equation*}
        \rho \mapsto (\rho_\Phi, N_\rho)
    \end{equation*}
    induces an equivalence of categories
    \begin{equation*}
        \operatorname{Rep}_{\overline{\Q}_\ell}(W_F) \rightarrow \operatorname{WD-Rep}_{\overline{\Q}_\ell}(W_F).
    \end{equation*}
    The isomorphism class of the Weil--Deligne representation $(\rho_\Phi, N_\rho)$ depends only on the isomorphism class of $\rho$ and not on the choices of $\Phi$ and $t_\ell$. Further, this assignment identifies semisimple representations of $W_F$ with $\Phi$-semisimple Weil--Deligne representations.
\end{theorem}

\begin{proof}
    See Bushnell and Henniart \cite[32.6 Theorem, 32.7 Theorem]{Bushnell-Henniart}.
\end{proof}

Since our goal is to go through each type of Weil--Deligne representation and study the corresponding continuous representation of $W_F$, we rewrite \eqref{Weil--Deligne correspondence bijection} as
\begin{equation*}
    \rho(\Phi^k \tau) = \rho_\Phi(\Phi^k \tau)\exp(t_\ell(\tau)N_\rho).
\end{equation*}
Note that we have $\rho = \rho_\Phi$ for principal series and nonexceptional supercuspidal representations since $N_\rho = 0$. Thus, we will abuse notation slightly and call $\rho$ a principal series or a nonexceptional supercuspidal representation to mean that the associated Weil--Deligne representation $(\rho_\Phi, 0) = (\rho, 0)$ is a principal series or a nonexceptional supercuspidal representation. It is only for special representations that we must be careful to distinguish between $\rho$ and $\rho_\Phi$ as special representations have a nonzero nilpotent operator.

We start with reductions of principal series:

\begin{theorem}
\label{thm:decomposability criterion for reductions of principal series}
    Let $\rho:W_F \rightarrow \GL_2(\overline{\Q}_\ell)$ be a principal series such that $\rho(W_F) \subset \GL_2(\calO_K)$. Then $\rho \cong \chi_1 \oplus \chi_2$, where $\chi_1, \chi_2$ are $1$-dimensional representations of $W_F$. Using Iwasawa decomposition, we can assume without loss of generality that there exists an upper triangular matrix
    \begin{equation*}
        U=\begin{pmatrix}
            a & b \\
            0 & d
        \end{pmatrix} \in \GL_2(K)
    \end{equation*}
    such that
    \begin{equation*}
        \rho(g) = U\begin{pmatrix}
            \chi_1(g) & 0 \\
            0 & \chi_2(g)
        \end{pmatrix}U^{-1}
    \end{equation*}
    for all $g \in W_F$. Then $\overline{\rho}$ is reducible. Further,

    \begin{enumerate}[(i)]
        \item if $\chi_1 = \chi_2$, then $\overline{\rho}$ decomposes as $\overline{\chi}_1 \oplus \overline{\chi}_2$.

        \item if $\overline{\chi}_1 \neq \overline{\chi}_2$, then $\overline{\rho}$ decomposes as $\overline{\chi}_1 \oplus \overline{\chi}_2$.

        \item if $\chi_1 \neq \chi_2$ and $\overline{\chi}_1 = \overline{\chi}_2$, then $\overline{\rho}$ is decomposable (and decomposes as $\overline{\chi}_1 \oplus \overline{\chi}_2$) if and only if $v(b/d) \geq 1-i$, where $i$ is the maximal nonnegative integer such that
        \begin{equation*}
            \chi_1(g) \equiv \chi_2(g) \pmod{\lambda^i}
        \end{equation*}
        for all $g \in W_F$.
    \end{enumerate}

    Equivalently, $\overline{\rho}$ is indecomposable if and only if $\chi_1 \neq \chi_2, \overline{\chi}_1 = \overline{\chi}_2$, and $v(b/d) = -i$.
\end{theorem}

\begin{proof}
    See Suözer \cite[Theorem 4.1.7]{Deniz}.
\end{proof}

Next, we look at reductions of nonexceptional supercuspidal representations. Unlike reductions of principal series, those of nonexceptional supercuspidal representations need not be reducible; we start by characterizing the conditions under which they are irreducible. Perhaps unsurprisingly, $\overline{\rho}$ is irreducible if and only if $\overline{\chi} \neq \overline{\chi}^\sigma$:

\begin{proposition}
\label{prop: irreducibility criterion for reductions of supercuspidal representations}
    Let $\rho:W_F \rightarrow \GL_2(\overline{\Q}_\ell)$ be a nonexceptional supercuspidal representation such that $\rho(W_F) \subset \GL_2(\calO_K)$. Then $\rho \cong \Ind_{W_E}^{W_F}\xi$, where $\xi$ is a $1$-dimensional representation of $W_E$ such that $\xi \neq \xi^\sigma$. Using Iwasawa decomposition, can assume without loss of generality that there exists an upper triangular matrix
    \begin{equation*}
        U = \begin{pmatrix}
            a & b \\
            0 & d
        \end{pmatrix} \in \GL_2(K)
    \end{equation*}
    such that
    \begin{equation*}
        \rho(g) = U\Ind_{W_E}^{W_F}\xi(g) U^{-1}
    \end{equation*}
    for all $g \in W_F$.

    Then $\overline{\rho}$ is irreducible if and only if $\overline{\xi} \neq \overline{\xi}^\sigma$. When $\overline{\rho}$ is irreducible, we have $d/a \in \calO_K^\times$ and $\overline{\rho} \cong \Ind_{W_E}^{W_F} \overline{\xi}$.
\end{proposition}

\begin{proof}
    See Suözer \cite[Proposition 4.1.9]{Deniz}.
\end{proof}

For the rest of our analysis of reductions of nonexceptional supercuspidal representations, we will focus on the case where $\overline{\rho}$ is reducible. Our goal is to apply \Cref{thm: decomposability criterion for reducible representations} to $\overline{\rho}$ to obtain a decomposability criterion. As we have seen in the proof of \Cref{prop: irreducibility criterion for reductions of supercuspidal representations}, we have $d/a \in \calO_K$ . The cases $d/a \in \calO_K^\times$ and $d/a \in \lambda$ require separate treatment.

If $d/a \in \calO_K^\times$, then $\overline{\rho}$ is an induced representation:

\begin{proposition}
\label{supercuspidal reduction: unit d/a implies induced representation}
    Let $\rho:W_F \rightarrow \GL_2(\overline{\Q}_\ell)$ be a nonexceptional supercuspidal representation such that $\rho(W_F) \subset \GL_2(\calO_K)$. Then $\rho \cong \Ind_{W_E}^{W_F}\xi$, where $\xi$ is a $1$-dimensional representation of $W_E$ such that $\xi \neq \xi^\sigma$. Using Iwasawa decomposition, we can assume without loss of generality that there exists an upper triangular matrix
    \begin{equation*}
        U = \begin{pmatrix}
            a & b \\
            0 & d
        \end{pmatrix} \in \GL_2(K)
    \end{equation*}
    such that
    \begin{equation*}
        \rho(g) = U\Ind_{W_E}^{W_F}\xi(g) U^{-1}
    \end{equation*}
    for all $g \in W_F$.

    Suppose that $\overline{\rho}$ is reducible and $d/a \in \calO_K^\times$. Then $\overline{\rho} \cong \Ind_{W_E}^{W_F}\overline{\xi}$, which decomposes as $\overline{\theta}_1 \oplus \overline{\theta}_2$, where
    \begin{alignat*}{2}
        \theta_1(g) &= \begin{cases} \xi(h), \\ \sqrt{\xi(\sigma^2)}\xi(h), \end{cases} 
        &\quad& \begin{aligned} & g=h \in W_E \\ & g = h \sigma, h \in W_E \end{aligned} \\[1.5ex]
        \theta_2(g) &= \begin{cases} \xi(h), \\ -\sqrt{\xi(\sigma^2)}\xi(h), \end{cases} 
        &\quad& \begin{aligned} & g=h \in W_E \\ & g = h \sigma, h \in W_E \end{aligned}
    \end{alignat*}
\end{proposition}

\begin{proof}
    See Suözer \cite[Proposition 4.1.11]{Deniz}.
\end{proof}

The case $d/a \in \lambda$ requires a more detailed analysis. We start with the following observation:

\begin{lemma}
\label{lem: d/a in lambda implies unit b/a for supercuspidal reductions}
    With the same setup as in \Cref{supercuspidal reduction: unit d/a implies induced representation}, suppose this time that $d/a \in \lambda$. Then $b/a \in \calO_K^\times$.
\end{lemma}

\begin{proof}
    See Suözer \cite[Lemma 4.1.12]{Deniz}.
\end{proof}

When $d/a \in \lambda$, $\overline{\rho}(g)$ is upper triangular for all $g \in W_F$. Letting
\begin{alignat*}{2}
        \theta_1(g) &= \begin{cases} \xi(h), \\ \frac{b}{a}\xi(\sigma h \sigma), \end{cases} 
        &\quad& \begin{aligned} & g=h \in W_E \\ & g = h \sigma, h \in W_E \end{aligned} \\[1.5ex]
        \theta_2(g) &= \begin{cases} \xi^\sigma(h), \\ -\frac{b}{a}\xi(\sigma h \sigma), \end{cases} 
        &\quad& \begin{aligned} & g=h \in W_E \\ & g = h \sigma, h \in W_E \end{aligned} \\[1.5ex]
        \psi(g) &= \begin{cases} \frac{b}{d}\left(\xi^\sigma(h)-\xi(h)\right), \\ \frac{a}{d}\xi(h)-\frac{b^2}{ad}\xi(\sigma h \sigma), \end{cases}
        &\quad& \begin{aligned} & g=h \in W_E \\ & g=h\sigma, h \in W_E
        \end{aligned}
\end{alignat*}
we can write $\overline{\rho}$ as

\begin{equation*}
    \overline{\rho}(g) = \begin{pmatrix}
        \overline{\theta}_1(g) & \overline{\psi}(g) \\
        0 & \overline{\theta}_2(g)
    \end{pmatrix}
\end{equation*}
for all $g \in W_F$. Then by \Cref{thm: decomposability criterion for reducible representations}, $\overline{\rho}$ is decomposable if and only if $\overline{\psi}$ is a scalar multiple of $\overline{\theta}_2 - \overline{\theta}_1$.

\begin{proposition}
\label{thm: decomposability criterion for nonexceptional supercuspidal representations for odd ell}
    Let $\rho:W_F \rightarrow \GL_2(\overline{\Q}_\ell)$ be a nonexceptional supercuspidal representation such that $\rho(W_F) \subset \GL_2(\calO_K)$. Then $\rho \cong \Ind_{W_E}^{W_F}\xi$, where $\xi$ is a $1$-dimensional representation of $W_E$ such that $\xi \neq \xi^\sigma$. Using Iwasawa decomposition, we can assume without loss of generality that there exists an upper triangular matrix
    \begin{equation*}
        U = \begin{pmatrix}
            a & b \\
            0 & d
        \end{pmatrix} \in \GL_2(K)
    \end{equation*}
    such that
    \begin{equation*}
        \rho(g) = U\Ind_{W_E}^{W_F}\xi(g) U^{-1}
    \end{equation*}
    for all $g \in W_F$.

    Suppose $d/a \in \lambda$ so that $\overline{\rho}$ is reducible by \Cref{prop: irreducibility criterion for reductions of supercuspidal representations}. Then $\overline{\rho}$ is decomposable if and only if $v(b/d) \geq 1-i$, where $i$ is the maximal nonnegative integer such that
    \begin{equation*}
        \xi(h) \equiv \xi^\sigma(h) \pmod{\lambda^i}
    \end{equation*}
    for all $h \in W_E$. If $\overline{\rho}$ is decomposable, it decomposes as $\overline{\theta}_1 \oplus \overline{\theta}_2$, where
    \begin{alignat*}{2}
        \theta_1(g) &= \begin{cases} \xi(h), \\ \frac{b}{a}\xi(\sigma h \sigma), \end{cases} 
        &\quad& \begin{aligned} & g=h \in W_E \\ & g = h \sigma, h \in W_E \end{aligned} \\[1.5ex]
        \theta_2(g) &= \begin{cases} \xi^\sigma(h), \\ -\frac{b}{a}\xi(\sigma h \sigma), \end{cases} 
        &\quad& \begin{aligned} & g=h \in W_E \\ & g = h \sigma, h \in W_E \end{aligned}
    \end{alignat*}
\end{proposition}

\begin{proof}
    See Suözer \cite[Proposition 4.1.13]{Deniz}.
\end{proof}

\begin{remark}
\label{cohomology supercusp reduction odd ell decomposability equiv to restriction decomposability}
    When $\overline{\xi} = \overline{\xi}^\sigma$, the condition $v(b/d) \geq 1-i$ is equivalent to the decomposability of $\overline{\rho}|_{W_E}$ by \Cref{thm:decomposability criterion for reductions of principal series}. Then \Cref{thm: decomposability criterion for nonexceptional supercuspidal representations for odd ell} states that, in the case $d/a \in \lambda$, $\overline{\rho}$ is decomposable if and only if $\overline{\rho}|_{W_E}$ is decomposable.
\end{remark}

We summarize our results for reductions of nonexceptional supercuspidal representations in the following theorem:

\begin{theorem}
\label{thm: decomposability criterion for nonexceptional supercuspidal representations}
    Let $\rho:W_F \rightarrow \GL_2(\overline{\Q}_\ell)$ be a nonexceptional supercuspidal representation such that $\rho(W_F) \subset \GL_2(\calO_K)$. Then $\rho \cong \Ind_{W_E}^{W_F}\xi$, where $\xi$ is a $1$-dimensional representation of $W_E$ such that $\xi \neq \xi^\sigma$. Using Iwasawa decomposition, we can assume without loss of generality that there exists an upper triangular matrix
    \begin{equation*}
        U = \begin{pmatrix}
            a & b \\
            0 & d
        \end{pmatrix} \in \GL_2(K)
    \end{equation*}
    such that
    \begin{equation*}
        \rho(g) = U\Ind_{W_E}^{W_F}\xi(g) U^{-1}
    \end{equation*}
    for all $g \in W_F$. Let $i$ be the maximal nonnegative integer such that
    \begin{equation*}
        \xi(h) \equiv \xi^\sigma(h) \pmod{\lambda^i}
    \end{equation*}
    for all $h \in W_E$.

    $\overline{\rho}$ is irreducible if and only if $\overline{\xi} \neq \overline{\xi}^\sigma$. If $\overline{\rho}$ is reducible, then:

    \begin{enumerate}[(i)]
        \item If $d/a \in \calO_K^\times$, then $\overline{\rho}$ is isomorphic to $\Ind_{W_E}^{W_F}\overline{\xi}$ and decomposes as $\overline{\theta}_1 \oplus \overline{\theta}_2$, where
        \begin{alignat*}{2}
            \theta_1(g) &= \begin{cases} \xi(h), \\ \sqrt{\xi(\sigma^2)}\xi(h), \end{cases} 
            &\quad& \begin{aligned} & g=h \in W_E \\ & g = h \sigma, h \in W_E \end{aligned} \\[1.5ex]
            \theta_2(g) &= \begin{cases} \xi(h), \\ -\sqrt{\xi(\sigma^2)}\xi(h), \end{cases} 
            &\quad& \begin{aligned} & g=h \in W_E \\ & g = h \sigma, h \in W_E \end{aligned}
        \end{alignat*}

        \item If $d/a \in \lambda$, then $\overline{\rho}$ is decomposable if and only if $v(b/d) \geq 1-i$, or equivalently, if $\overline{\rho}|_{W_E}$ is decomposable. When $\overline{\rho}$ is decomposable, it decomposes as $\overline{\theta}_1 \oplus \overline{\theta}_2$, where
        \begin{alignat*}{2}
            \theta_1(g) &= \begin{cases} \xi(h), \\ \frac{b}{a}\xi(\sigma h \sigma), \end{cases} 
            &\quad& \begin{aligned} & g=h \in W_E \\ & g = h \sigma, h \in W_E \end{aligned} \\[1.5ex]
            \theta_2(g) &= \begin{cases} \xi^\sigma(h), \\ -\frac{b}{a}\xi(\sigma h \sigma), \end{cases} 
            &\quad& \begin{aligned} & g=h \in W_E \\ & g = h \sigma, h \in W_E \end{aligned}
        \end{alignat*}
    \end{enumerate}
\end{theorem}

\begin{proof}
    By \Cref{prop: irreducibility criterion for reductions of supercuspidal representations}, $\overline{\rho}$ is irreducible if and only if $\overline{\xi} \neq \overline{\xi}^\sigma$. Suppose that $\overline{\rho}$ is reducible. If $d/a \in \calO_K^\times$, then $\overline{\rho}$ is isomorphic to $\Ind_{W_E}^{W_F}\overline{\xi}$ and decomposes as $\overline{\theta}_1 \oplus \overline{\theta}_2$ by \Cref{supercuspidal reduction: unit d/a implies induced representation}. If $d/a \in \lambda$, then by \Cref{thm: decomposability criterion for nonexceptional supercuspidal representations for odd ell}, $\overline{\rho}$ is decomposable if and only if $v(b/d) \geq 1-i$, in which case it decomposes as $\overline{\theta}_1 \oplus \overline{\theta}_2$. As we have seen in \Cref{cohomology supercusp reduction odd ell decomposability equiv to restriction decomposability}, the condition $v(b/d) \geq 1-i$ is equivalent to the decomposability of $\overline{\rho}|_{W_E}$.
\end{proof}

Lastly, we look at reductions of representations associated with special representations. We describe what these representations look like in the following lemma:

\begin{lemma}
\label{special representation matrix form}
    Let $\rho:W_F \rightarrow \GL_2(\overline{\Q}_\ell)$ be a continuous, semisimple representation such that $\rho(W_F) \subset \GL_2(\calO_K)$. Suppose that the associated Weil--Deligne representation $(\rho_\Phi, N_\rho)$ is a special representation so that $\rho_\Phi \cong \chi\omega \oplus \chi$, where $\chi$ is a $1$-dimensional representation of $W_F$. Then there exists a continuous, semisimple representation $\eta$ of $W_F$ defined by
    \begin{equation}
    \label{special rep matrix form: the explicit form}
        \eta(\Phi^k\tau) = \begin{pmatrix}
            \chi\omega(\Phi^k\tau) & \frac{a}{d}t_\ell(\tau)\chi\omega(\Phi^k\tau) + \frac{b}{d}(\chi(\Phi^k\tau) - \chi\omega(\Phi^k\tau)) \\
            0 & \chi(\Phi^k\tau)
        \end{pmatrix},
    \end{equation}
    where $a, b, d \in K$ and $ad \neq 0$ such that $\overline{\rho} \cong \overline{\eta}$.
\end{lemma}

\begin{proof}
    As $(\rho_\Phi, N_\rho)$ is a special representation, there exists $M \in \GL_2(K)$ such that
    \begin{align}
        M\begin{pmatrix}
            \chi\omega(g) & 0 \\
            0 & \chi(g)
        \end{pmatrix}M^{-1} &= \rho_\Phi(g) \label{special representation matrix form: associated w-d conjugated form} \\
        M\begin{pmatrix}
            0 & 1 \\
            0 & 0
        \end{pmatrix}M^{-1} &= N_\rho \label{special representation matrix form: associated nilpotent operator form}
    \end{align}
    for all $g \in W_F$. By \Cref{Weil--Deligne and continuous weil representation correspondence}, we have the relation
    \begin{equation}
    \label{special representation matrix form: grothendieck correspondence}
        \rho(\Phi^k\tau) = \rho_\Phi(\Phi^k\tau)\exp(t_\ell(\tau)N_\rho)
    \end{equation}
    between $\rho, \rho_\Phi$, and $N_\rho$. Plugging \eqref{special representation matrix form: associated w-d conjugated form} and \eqref{special representation matrix form: associated nilpotent operator form} into \eqref{special representation matrix form: grothendieck correspondence}, we obtain
    \begin{equation*}
        \rho(\Phi^k\tau) = M\begin{pmatrix}
            \chi\omega(\Phi^k\tau) & 0 \\
            0 & \chi(\Phi^k\tau)
        \end{pmatrix}M^{-1} \exp \left(t_\ell(\tau)M\begin{pmatrix}
            0 & 1 \\
            0 & 0
        \end{pmatrix}M^{-1}  \right).
    \end{equation*}
    Using the Taylor series for $\exp$, we further simplify this to
    \begin{equation*}
        \rho(\Phi^k\tau) = M \left( \begin{pmatrix}
            \chi\omega(\Phi^k\tau) & 0 \\
            0 & \chi(\Phi^k\tau)
        \end{pmatrix} + t_\ell(\tau)\begin{pmatrix}
            \chi\omega(\Phi^k\tau) & 0 \\
            0 & \chi(\Phi^k\tau)
        \end{pmatrix}\begin{pmatrix}
            0 & 1 \\
            0 & 0
        \end{pmatrix}\right)M^{-1}.
    \end{equation*}
    We can decompose $M$ as $M = AU$ using Iwasawa decomposition, where $U \in \GL_2(K)$ is upper triangular and $A \in \GL_2(\calO_K)$. Letting
    \begin{equation*}
        U = \begin{pmatrix}
            a & b \\
            0 & d
        \end{pmatrix},
    \end{equation*}
    we get
    \begin{align*}
        \rho(\Phi^k\tau) &= AU \left( \begin{pmatrix}
            \chi\omega(\Phi^k\tau) & 0 \\
            0 & \chi(\Phi^k\tau)
        \end{pmatrix} + t_\ell(\tau)\begin{pmatrix}
            \chi\omega(\Phi^k\tau) & 0 \\
            0 & \chi(\Phi^k\tau)
        \end{pmatrix}\begin{pmatrix}
            0 & 1 \\
            0 & 0
        \end{pmatrix}\right)U^{-1}A^{-1} \\
        &= A\begin{pmatrix}
            \chi\omega(\Phi^k\tau) & \frac{a}{d}t_\ell(\tau)\chi\omega(\Phi^k\tau) + \frac{b}{d}(\chi(\Phi^k\tau) - \chi\omega(\Phi^k\tau)) \\
            0 & \chi(\Phi^k\tau)
        \end{pmatrix} A^{-1}.
    \end{align*}
    Hence, $\overline{\rho} \cong \overline{\eta}$.
\end{proof}

As reductions of the representations are what we are interested in, we can assume without loss of generality that $\rho$ is defined by \eqref{special rep matrix form: the explicit form}.

The following lemma shows that when $a/d \in \lambda$, the reduction $\overline{\rho}$ is isomorphic to the reduction of a principal series:

\begin{lemma}
\label{special rep principal reduction case}
    With the same setup as in \Cref{special representation matrix form}, suppose that $a/d \in \lambda$. Then $\overline{\rho} = \overline{\eta}$, where $\eta:W_F \rightarrow \GL_2(\overline{\Q}_\ell)$ is a continuous, semisimple representation such that the associated Weil--Deligne representation $(\eta_\Phi, N_\rho)$ is the principal series $(\chi\omega\oplus\chi, 0)$. 
\end{lemma}

\begin{proof}
    By \Cref{special representation matrix form}, we can assume without loss of generality that $\rho$ is defined by
    \begin{equation*}
        \rho(\Phi^k\tau) = \begin{pmatrix}
        \chi\omega(\Phi^k\tau) & \frac{a}{d}t_\ell(\tau)\chi\omega(\Phi^k\tau)+\frac{b}{d}(\chi(\Phi^k\tau)-\chi\omega(\Phi^k\tau)) \\
        0 & \chi(\Phi^k\tau)
    \end{pmatrix}.
    \end{equation*}
    Note that $\chi\omega(W_F) \subset \calO_K^\times$ and $t_\ell(\tau) \in \calO_K$ for all $\tau \in I_F$, so the term
    \begin{equation*}
        \frac{a}{d}t_\ell(\tau)\chi\omega(\Phi^k\tau)
    \end{equation*}
    vanishes mod $\lambda$ as $a/d \in \lambda$ by assumption and we are left with
    \begin{equation*}
        \overline{\rho}(\Phi^k\tau) = \begin{pmatrix}
            \chi\omega(\Phi^k\tau) & \frac{b}{d}(\chi(\Phi^k\tau) - \chi\omega(\Phi^k\tau)) \\
            0 & \chi(\Phi^k\tau)
        \end{pmatrix}.
    \end{equation*}
    Observe that
    \begin{equation*}
        \begin{pmatrix}
            \chi\omega(\Phi^k\tau) & \frac{b}{d}(\chi(\Phi^k\tau) - \chi\omega(\Phi^k\tau)) \\
            0 & \chi(\Phi^k\tau)
        \end{pmatrix} = U\begin{pmatrix}
            \chi\omega(\Phi^k\tau) & 0 \\
            0 & \chi(\Phi^k\tau)
        \end{pmatrix}U^{-1},
    \end{equation*}
    where
    \begin{equation*}
        U = \begin{pmatrix}
            a & b \\
            0 & d
        \end{pmatrix} \in \GL_2(K).
    \end{equation*}
    Hence, $\overline{\rho} = \overline{\eta}$, where $\eta$ is the principal series defined by
    \begin{equation*}
        \eta(g) = U\begin{pmatrix}
            \chi\omega(g) & 0 \\
            0 & \chi(g)
        \end{pmatrix}U^{-1}.
    \end{equation*}
\end{proof}

\Cref{special rep principal reduction case} lets us to restrict our attention to the case $a/d \in \calO_K^\times$ when comparing reductions of special representations with other types. If $a/d \in \lambda$, then $\overline{\rho}$ is isomorphic to the reduction of a principal series, so the comparison of $\overline{\rho}$ with another representation reduces to the corresponding comparison for principal series, which is treated separately.

If $a/d \in \calO_K^\times$, then $\overline{\rho}$ is indecomposable:

\begin{lemma}
\label{lem: special representation decomposability condition}
    With the same setup as in \Cref{special representation matrix form}, suppose that $\overline{\rho}$ is decomposable. Then $a/d \in \lambda$.
\end{lemma}

\begin{proof}
    By \Cref{special representation matrix form}, we can assume without loss of generality that $\rho$ is defined by
    \begin{equation*}
        \rho(\Phi^k\tau) = \begin{pmatrix}
        \chi\omega(\Phi^k\tau) & \frac{a}{d}t_\ell(\tau)\chi\omega(\Phi^k\tau)+\frac{b}{d}(\chi(\Phi^k\tau)-\chi\omega(\Phi^k\tau)) \\
        0 & \chi(\Phi^k\tau)
    \end{pmatrix},
    \end{equation*}
    where $a, b, d \in K$ and $ad \neq 0$. Since $\overline{\rho}$ is decomposable, by \Cref{thm: decomposability criterion for reducible representations}, there exists $c \in \calO_K$ such that
    \begin{equation}
    \label{special rep decomposability condition cong}
        c(\chi(\Phi^k\tau)-\chi\omega(\Phi^k\tau)) \equiv \frac{a}{d}t_\ell(\tau)\chi\omega(\Phi^k\tau)+\frac{b}{d}(\chi(\Phi^k\tau)-\chi\omega(\Phi^k\tau)) \pmod{\lambda}
    \end{equation}
    for all $\tau \in I_F, k \in \Z$. Letting $k = 0$, \eqref{special rep decomposability condition cong} becomes
    \begin{equation*}
        0 \equiv \frac{a}{d}t_\ell(\tau)\chi(\tau) \pmod{\lambda}
    \end{equation*}
    for all $\tau \in I_F$. As we can pick $\tau_0 \in I_F$ such that $t_\ell(\tau_0) \in \calO_K^\times$, we conclude that $a/d \in \lambda$.
\end{proof}

If $a/d \in \lambda$, $\overline{\rho}$ is isomorphic to the reduction of a principal series, so whether or not $\overline{\rho}$ is decomposable is completely described by \Cref{thm:decomposability criterion for reductions of principal series}. Otherwise, if $a/d \in \calO_K^\times$, then $\overline{\rho}$ is indecomposable by \Cref{lem: special representation decomposability condition}. This concludes our analysis of reductions of representations of $W_F$ associated with special representations. Before summarizing our description, we prove the following lemma, which we will use to simplify our subsequent decomposability description:

\begin{lemma}
\label{lem: omega trivial reduction criterion}
    The $1$-dimensional representation $\omega:W_F \rightarrow K^\times$ acts as the identity on the residue field $\calO_K/\lambda$ if and only if $q \equiv 1 \pmod{\ell}$.
\end{lemma}

\begin{proof}
    For any $\Phi^k\tau \in W_F$, we have $\omega(\Phi^k\tau)=q^{-k}$. Suppose that $\overline{\omega} = \overline{1}$. Then $q^{-k} \equiv 1 \pmod{\lambda}$ for all $k \in \Z$. Since $q^{-k} \in \Q$, this congruence is equivalent to $q^{-k}\equiv 1 \pmod{\ell}$. Letting $k = 0$, we obtain $q \equiv 1 \pmod{\ell}$.

    Conversely, suppose that $q \equiv 1 \pmod{\ell}$. Since $q-1$ divides $q^k-1$ for all $k \in \mathbb{N}$, we have $q^k \equiv 1 \pmod{\ell}$ for all $k \in \mathbb{N}$. To conclude, note that for $k \in \mathbb{N}$, then $q^{-k}-1 = \frac{1-q^k}{q^k}$, so $v(q^{-k}-1) = v(1-q^k)$ as $q \in \calO_K^\times$. Hence, $\overline{\omega} = \overline{1}$.
\end{proof}

\begin{theorem}
\label{thm:decomposability criterion for reductions of special representations}
    Let $\rho:W_F \rightarrow \GL_2(\overline{\Q}_\ell)$ be a continuous, semisimple representation such that $\rho(W_F) \subset \GL_2(\calO_K)$. Suppose that the associated Weil--Deligne representation $(\rho_\Phi, N_\rho)$ is a special representation. By \Cref{special representation matrix form}, we can assume without loss of generality that $\rho$ is defined by
    \begin{equation*}
        \rho(\Phi^k\tau) = \begin{pmatrix}
            \chi\omega(\Phi^k\tau) & \frac{a}{d}t_\ell(\tau)\chi\omega(\Phi^k\tau) + \frac{b}{d}(\chi\omega(\Phi^k\tau) - \chi(\Phi^k\tau) \\
            0 & \chi(\Phi^k\tau)
        \end{pmatrix},
    \end{equation*}
    where $a, b, d \in K$ and $ad \neq 0$.

    The reduction  $\overline{\rho}$ is reducible. Further:

    \begin{enumerate}[(i)]
        \item If $a/d \in \calO_K^\times$, $\overline{\rho}$ is indecomposable.
        \item If $q \not\equiv 1 \pmod{\ell}$ and $a/d \in \lambda$, then $\overline{\rho}$ decomposes as $\overline{\chi\omega} \oplus \overline{\chi}$.

        \item If $q \equiv 1 \pmod{\ell}$ and $a/d \in \lambda$, then $\overline{\rho}$ is decomposable (and decomposes as $\overline{\chi\omega} \oplus \overline{\chi}$) if and only if $v(b/d) \geq 1-ei$, where $e$ is the ramification index of the extension $K/\Q_{\ell}$ and $i$ is the maximal nonnegative integer such that
        \begin{equation*}
            q \equiv 1 \pmod{\ell^i}
        \end{equation*}
        for all $g \in W_F$.
    \end{enumerate}
\end{theorem}

\begin{proof}
    If $a/d \in \calO_K^\times$, then $\overline{\rho}$ is indecomposable by \Cref{lem: special representation decomposability condition}. Suppose now that $a/d \in \lambda$. Then by \Cref{special rep principal reduction case}, $\overline{\rho}$ is the reduction of the principal series $\chi\omega\oplus\chi$. We can then assume without loss of generality that $\rho$ is a principal series and apply \Cref{thm:decomposability criterion for reductions of principal series}.
    
    If $\overline{\chi\omega} \neq \overline{\chi}$, then $\overline{\rho}$ is decomposable. Note that this is equivalent to $\overline{\omega} \neq \overline{1}$, which is further equivalent to $q \equiv 1 \pmod{\ell}$ by \Cref{lem: omega trivial reduction criterion}.

    Lastly, if $\overline{\chi\omega} = \overline{\chi}$ (or equivalently, $q \equiv 1 \pmod{\ell}$), then $\overline{\rho}$ is decomposable if and only if $v(b/d) \geq 1-j$, where $j$ is the maximal nonnegative integer such that
    \begin{equation*}
        \chi\omega(g) \equiv \chi(g) \pmod{\lambda^j}
    \end{equation*}
    for all $g \in W_F$. Equivalently,
    \begin{equation*}
        \omega(\Phi^k\tau) = q^{-k} \equiv 1 \pmod{\lambda^j}
    \end{equation*}
    for all $k \in \Z$, which in turn is equivalent to
    \begin{equation*}
        q \equiv 1 \pmod{\lambda^j},
    \end{equation*}
    i.e., $v_K(q-1) = j$. Recall that for any $x\in \Z_{\ell}$, we have $v_K(x) = ev_\ell(x)$, where $e$ is the ramification index of $K/\Q_{\ell}$. Setting $i := v_\ell(q-1)$, we obtain the desired result.
\end{proof}

\section{Comparison of Reductions of $\ell$-adic Representations of the Weil Group}
\label{comparison section}

Recall that $p$ and $\ell$ are distinct, odd primes. We denote by $F$ a finite extension of $\Q_p$ with ring of integers $\calO_F$, maximal ideal $\p$, and residue field $k_F$ of order $q$. By $E, E_1, E_2$, we denote quadratic extensions of $F$, with corresponding rings of integers $\calO_E, \calO_{E_1}, \calO_{E_2}$ and residue fields $k_E, k_{E_1}, k_{E_2}$. We denote the absolute Galois group, the Weil group, the inertia group, and the wild inertia group of $F$ by $G_F, W_F, I_F$, and $P_F$, respectively.

Let $K$ be a finite extension of $\Q_\ell$, with ring of integers $\calO_K$, maximal ideal $\lambda$, and residue field $k_K$.

By $\omega$, we denote the $1$-dimensional representation of $W_F$ corresponding to the absolute value on $F^\times$ via local class field theory. Fixing a Frobenius element $\Phi$, $\omega$ is explicitly given by $\omega(\Phi^k\tau) = q^{-k}$ for any $\tau \in I_F, k \in \Z$. We denote by $t_\ell: I_F \rightarrow \Z_\ell$ a fixed continuous, surjective homomorphism.

When considering a supercuspidal representation $\rho$ induced from $W_E$ (likewise, $W_{E_1}$ or $W_{E_2}$), we use $\sigma \in W_F \setminus W_E$ to denote a fixed representative of the nontrivial coset of $W_E$ in $W_F$.

In this section, we will determine when the reductions of two semisimple, continuous $\ell$-adic representations $\rho_1, \rho_2:W_F \rightarrow \GL_2(\overline{\Q}_\ell)$ are isomorphic, where $\rho_1$ and $\rho_2$ are principal series, nonexceptional supercuspidal representations, or special representations.

It is relatively straightforward to compare reductions that are both irreducible or both decomposable. Thus, the main goal of this section is to describe when two reducible but indecomposable reductions are isomorphic. For this, we will make use of the following, more general, theorem:

\begin{theorem}
\label{thm: isomorphism of indecomposable representations}
    Let $G$ be a group and $\rho_1, \rho_2:G \rightarrow \GL_2(L)$ be two indecomposable representations over a field $L$ that are defined by
    \begin{align*}
        \rho_1(g) &= \begin{pmatrix}
            \chi_1(g) & \psi_1(g) \\
            0 & \chi_2(g)
        \end{pmatrix} \\
        \rho_2(g) &= \begin{pmatrix}
            \theta_1(g) & \psi_2(g) \\
            0 & \theta_2(g)
        \end{pmatrix}
    \end{align*}
    for all $g \in G$. Then $\rho_1 \cong \rho_2$ if and only if $\chi_1 = \theta_1, \chi_2 = \theta_2$, and there exist $x \in L^\times, y \in L$ such that
    \begin{equation}
    \label{cohomology indecomposable rep equivalence eqn}
        x\psi_1+y(\chi_2-\chi_1) = \psi_2.
    \end{equation}
\end{theorem}

\begin{proof}
    See Suözer \cite[Proposition 2.4.24]{Deniz}.
\end{proof}

We begin by comparing the reductions of two principal series:

\begin{theorem}
\label{thm: principal vs principal}
    Let $\rho_1, \rho_2:W_F \rightarrow \GL_2(\overline{\Q}_\ell)$ be two principal series such that $\rho_1(W_F), \rho_2(W_F) \subset \GL_2(\calO_K)$. Then $\rho_1 \cong \chi_1 \oplus \chi_2$ and $\rho_2 \cong \theta_1 \oplus \theta_2$, where $\chi_1, \chi_2, \theta_1, \theta_2$ are $1$-dimensional representations of $W_F$. Using Iwasawa decomposition, we can assume without loss of generality that there exist upper triangular matrices
    \begin{equation*}
        U_1 = \begin{pmatrix}
            a_1 & b_1 \\
            0 & d_1
        \end{pmatrix}, U_2 = \begin{pmatrix}
            a_2 & b_2 \\
            0 & d_2
        \end{pmatrix}
    \end{equation*}
    in $\GL_2(K)$ such that
    \begin{align*}
        \rho_1(g) &= U_1\begin{pmatrix}
            \chi_1(g) & 0 \\
            0 & \chi_2(g)
        \end{pmatrix}U_1^{-1} \\
        \rho_2(g) &= U_2\begin{pmatrix}
            \theta_1(g) & 0 \\
            0 & \theta_2(g)
        \end{pmatrix}U_2^{-1}
    \end{align*}
    for all $g \in W_F$.

    \begin{enumerate}[(i)]
        \item If $\overline{\rho}_1$ and $\overline{\rho}_2$ are both decomposable, then $\overline{\rho}_1 \cong \overline{\rho}_2$ if and only if $\{\overline{\chi}_1, \overline{\chi}_2\} = \{\overline{\theta}_1, \overline{\theta}_2\}$.

        \item If $\overline{\rho}_1$ and $\overline{\rho}_2$ are both indecomposable, then $\overline{\chi}_1 = \overline{\chi}_2$ and $\overline{\theta}_1 = \overline{\theta}_2$ by \Cref{thm:decomposability criterion for reductions of principal series}. In that case, $\overline{\rho}_1 \cong \overline{\rho}_2$ if and only if $\overline{\chi}_1 = \overline{\theta}_1$ and there exists $x \in \calO_K^\times$ such that
        \begin{equation*}
            x\frac{b_1}{d_1}\left(\chi_2(g)-\chi_1(g) \right) \equiv \frac{b_2}{d_2}\left(\theta_2(g) - \theta_1(g) \right) \pmod{\lambda}
        \end{equation*}
        for all $g \in W_F$.
    \end{enumerate}
\end{theorem}

\begin{proof}
    See Suözer \cite[Theorem 4.2.1]{Deniz}.
\end{proof}

Next, we compare the reductions of principal series and nonexceptional supercuspidal representations:

\begin{theorem}
\label{cohomology principal vs supercuspidal comparison}
    Let $\rho_1, \rho_2:W_F \rightarrow \GL_2(\overline{\Q}_\ell)$ be a principal series and a nonexceptional supercuspidal representation, respectively, such that $\rho_1(W_F), \rho_2(W_F) \subset \GL_2(\calO_K)$. Then $\rho_1 \cong \chi_1 \oplus \chi_2$ and $\rho_2 \cong \Ind_{W_E}^{W_F}\xi$, where $\chi_1, \chi_2$ are $1$-dimensional representations of $W_F$ and $\xi$ is a $1$-dimensional representation of $W_E$. Using Iwasawa decomposition, we can assume without loss of generality that there exist upper triangular matrices
    \begin{equation*}
        U_1 = \begin{pmatrix}
            a_1 & b_1 \\
            0 & d_1
        \end{pmatrix}, U_2 = \begin{pmatrix}
            a_2 & b_2 \\
            0 & d_2
        \end{pmatrix}
    \end{equation*}
    in $\GL_2(K)$ such that
    \begin{align*}
        \rho_1(g) &= U_1\begin{pmatrix}
            \chi_1(g) & 0 \\
            0 & \chi_2(g)
        \end{pmatrix}U_1^{-1} \\
        \rho_2(g) &= U_2\Ind_{W_E}^{W_F}\xi(g) U_2^{-1}
    \end{align*}
    for all $g \in W_F$.

    Then $\overline{\rho}_1 \cong \overline{\rho}_2$ if and only if $\overline{\rho}_1$ and $\overline{\rho}_2$ are both decomposable, $\overline{\chi}_1|_{W_E} = \overline{\chi}_2|_{W_E} = \overline{\xi}$, and $\{\overline{\chi}_1(\sigma), \overline{\chi}_2(\sigma)\} = S$, where $S$ is defined as:

    \begin{enumerate}[(i)]
        \item If $d_2/a_2 \in \calO_K^\times$, then $S = \left\{\pm\overline{\sqrt{\xi(\sigma^2)}}\right\}$.

        \item If $d_2/a_2 \in \lambda$, then $S = \{\pm\overline{b}_2\overline{\xi}(\sigma^2)/\overline{a}_2\}$.
    \end{enumerate}
\end{theorem}

\begin{proof}
    See Suözer \cite[Theorem 4.2.2]{Deniz}.
\end{proof}

Lastly, we compare reductions of principal series and special representations. By \Cref{special rep principal reduction case}, it suffices to consider the case $a/d \in \calO_K^\times$ for special representations, since otherwise the reduction is isomorphic to the reduction of another principal series and the comparison reduces to \Cref{thm: principal vs principal}.

\begin{theorem}
\label{thm: principal vs special}
    Let $\rho_1, \rho_2:W_F \rightarrow \GL_2(\overline{\Q}_\ell)$ be a principal series and a special representation, respectively, such that $\rho_1(W_F), \rho_2(W_F) \subset \GL_2(\calO_K)$. Then $\chi_1\oplus \chi_2$ and $\chi\omega\oplus\chi$ are the Weil--Deligne representations associated with $\rho_1$ and $\rho_2$, respectively, where $\chi_1, \chi_2,\ \chi$ are $1$-dimensional representations of $W_F$. Using Iwasawa decomposition and \Cref{special representation matrix form}, we can assume without loss of generality that there exist upper triangular matrices
    \begin{equation*}
        U_1 = \begin{pmatrix}
            a_1 & b_1 \\
            0 & d_1
        \end{pmatrix}, U_2 = \begin{pmatrix}
            a_2 & b_2 \\
            0 & d_2
        \end{pmatrix}
    \end{equation*}
    in $\GL_2(K)$ such that
    \begin{align*}
        \rho_1(g) &= \begin{pmatrix}
            \chi_1(g) & \frac{b_1}{d_1}(\chi_2(g) - \chi_1(g)) \\
            0 & \chi_2(g)
        \end{pmatrix} \\
        \rho_2(\Phi^k\tau) &= \begin{pmatrix}
            \chi\omega(\Phi^k\tau) & \frac{a_2}{d_2}t_\ell(\tau)\chi\omega(\Phi^k\tau) + \frac{b_2}{d_2}(\chi(\Phi^k\tau) - \chi\omega(\Phi^k\tau)) \\
            0 & \chi(\Phi^k\tau)
        \end{pmatrix}.
    \end{align*}
    Suppose that $a_2/d_2 \in \calO_K^\times$. Then $\overline{\rho}_1 \cong \overline{\rho}_2$ if and only if $q \equiv 1 \pmod{\ell}$, $\overline{\chi} = \overline{\chi}_1 = \overline{\chi}_2$, and there exists $x \in \calO_K^\times$ such that
    \begin{equation}
    \label{principal vs special: congruence 1}
        \frac{b_1x}{d_1}(\chi_2(\Phi^k\tau) - \chi_1(\Phi^k\tau)) \equiv \frac{a_2}{d_2}t_\ell(\tau)\chi\omega(\Phi^k\tau) + \frac{b_2}{d_2}(\chi(\Phi^k\tau) - \chi\omega(\Phi^k\tau)) \pmod{\lambda}
    \end{equation}
    for all $\tau \in I_F, k \in \Z$.
\end{theorem}

\begin{proof}
    Since $a_2/d_2 \in \calO_K^\times$, $\overline{\rho}_2$ is indecomposable by \Cref{thm:decomposability criterion for reductions of special representations}, so $\overline{\rho}_1$ must also be indecomposable to be isomorphic to $\overline{\rho}_2$. By \Cref{thm:decomposability criterion for reductions of principal series}, this implies that $\overline{\chi}_1 = \overline{\chi}_2$.

    By \Cref{thm: isomorphism of indecomposable representations}, $\overline{\rho}_1 \cong \overline{\rho}_2$ if and only if $\overline{\chi\omega} = \overline{\chi}_1, \overline{\chi} = \overline{\chi}_2$, and there exist $x \in \calO_K^\times, y \in \calO_K$ such that
    \begin{equation}
    \label{principal vs special: congruence 2}
        \frac{b_1x}{d_1}(\chi_2(\Phi^k\tau) - \chi_1(\Phi^k\tau)) + y(\chi_2(\Phi^k\tau)-\chi_1(\Phi^k\tau)) \equiv \psi(\Phi^k\tau) \pmod{\lambda}
    \end{equation}
    for all $\tau \in I_F, k \in \Z$, where
    \begin{equation*}
        \psi(\Phi^k\tau) = \frac{a_2}{d_2}t_\ell(\tau)\chi\omega(\Phi^k\tau) + \frac{b_2}{d_2}(\chi(\Phi^k\tau) - \chi\omega(\Phi^k\tau)).
    \end{equation*}
    Observe that the term $y(\chi_2(\Phi^k\tau)-\chi_1(\Phi^k\tau))$ in \eqref{principal vs special: congruence 2} vanishes mod $\lambda$ since $\overline{\chi}_1 = \overline{\chi}_2$. Thus, we are left with
    \begin{equation}
        \frac{b_1x}{d_1}(\chi_2(\Phi^k\tau) - \chi_1(\Phi^k\tau)) \equiv \psi(\Phi^k\tau) \pmod{\lambda}.
    \end{equation}
    Plugging $\psi$ back in, we obtain \eqref{principal vs special: congruence 1}. To conclude, note that the identities $\overline{\chi\omega} = \overline{\chi}_1, \overline{\chi} = \overline{\chi}_2$ are equivalent to $\overline{\chi} = \overline{\chi}_1 = \overline{\chi}_2$ and $\overline{\omega} = \overline{1}$. By \Cref{lem: omega trivial reduction criterion}, this last identity is equivalent to $q \equiv 1 \pmod{\ell}$.
\end{proof}

Next, we compare the reductions of two nonexceptional supercuspidal representations:

\begin{theorem}
    Let $\rho_1, \rho_2 : W_F \rightarrow \GL_2(\overline{\Q}_\ell)$ be two nonexceptional supercuspidal representations such that $\rho_1(W_F), \rho_2(W_F) \subset \GL_2(\calO_K)$. Then $\rho_1 \cong \Ind_{W_{E_1}}^{W_F}\xi_1$ and $\rho_2 \cong \Ind_{W_{E_2} }^{W_F}\xi_2$, where $\xi_1, \xi_2$ are $1$-dimensional representations of $W_{E_1}, W_{E_2}$, respectively. Using Iwasawa decomposition, we can assume without loss of generality that there exist upper triangular matrices
    \begin{equation*}
        U_1 = \begin{pmatrix}
            a_1 & b_1 \\
            0 & d_1
        \end{pmatrix}, U_2 = \begin{pmatrix}
            a_2 & b_2 \\
            0 & d_2
        \end{pmatrix}
    \end{equation*}
    in $\GL_2(K)$ such that
    \begin{align*}
        \rho_1(g) &= U_1 \Ind_{W_{E_1}}^{W_F}\xi_1(g) U_1^{-1} \\
        \rho_2(g) &= U_2\Ind_{W_{E_2}}^{W_F}\xi_2(g) U_2^{-1}
    \end{align*}
    for all $g \in W_F$. Suppose further that $v(d_1/a_1) \leq v(d_2/a_2)$.

    \begin{enumerate}[(i)]
        \item If $\overline{\rho}_1$ and $\overline{\rho}_2$ are both irreducible and $E_1 = E_2$, then $\overline{\rho}_1 \cong \overline{\rho}_2$ if and only if $\overline{\xi}_1 = \overline{\xi}_2$ or $\overline{\xi}_1 = \overline{\xi}_2^\sigma$.

        \item If $\overline{\rho}_1$ and $\overline{\rho}_2$ are both irreducible and $E_1 \neq E_2$, then $\overline{\rho}_1 \cong \overline{\rho}_2$ if and only if
        \begin{equation*}
        \label{cohomology supercusp comparison summary thm eqn}
            \overline{\xi}_1|_{W_{E_1} \cap W_{E_2} } = \overline{\xi}_1^\sigma|_{W_{E_1} \cap W_{E_2} } = \overline{\xi}_2|_{W_{E_1} \cap W_{E_2} } = \overline{\xi}_2^\sigma|_{W_{E_1} \cap W_{E_2} }.
        \end{equation*}

        \item If $\overline{\rho}_1$ and $\overline{\rho}_2$ are both decomposable and $d_1/a_1, d_2/a_2 \in \calO_K^\times$, then $\overline{\rho}_1 \cong \overline{\rho}_2$ if and only if $E_1 = E_2$ and $\overline{\xi}_1 = \overline{\xi}_2$.

        \item If $\overline{\rho}_1$ and $\overline{\rho}_2$ are both decomposable and $d_1/a_1 \in \calO_K^\times, d_2/a_2 \in \lambda$, then $\overline{\rho}_1 \cong \overline{\rho}_2$ if and only if $E_1 = E_2$, $\overline{\xi}_1 = \overline{\xi}_2$, and
        \begin{equation*}
            \sqrt{\xi_1(\sigma^2)} \equiv \pm \frac{a_2}{b_2} \pmod{\lambda}.
        \end{equation*}

        \item If $\overline{\rho}_1$ and $\overline{\rho}_2$ are both decomposable and $d_1/a_1, d_2/a_2 \in \lambda$, then $\overline{\rho}_1 \cong \overline{\rho}_2$ if and only if $E_1 = E_2$, $\overline{\xi}_1 = \overline{\xi}_2$, and
        \begin{equation*}
            \frac{b_1}{a_1} \equiv \pm\frac{b_2}{a_2} \pmod{\lambda}.
        \end{equation*}

        \item If $\overline{\rho}_1$ and $\overline{\rho}_2$ are both reducible and indecomposable, then  $\overline{\rho}_1 \cong \overline{\rho}_2$ if and only if $E_1 = E_2 =: E, \overline{\rho}_1|_{W_E} \cong \overline{\rho}_2|_{W_E}$ and
        \begin{equation*}
        \label{cohomology supercusp comparison summary thm congruence}
            \frac{b_1}{a_1} \equiv \frac{b_2}{a_2} \pmod{\lambda}.
        \end{equation*}
    \end{enumerate}
\end{theorem}

\begin{proof}
    See Suözer \cite[Theorem 4.2.8]{Deniz}.
\end{proof}

Next, we compare reductions of supercuspidal representations with those of special representations. It turns out that under the assumption $a/d \in \calO_K^\times$ for special representations, the reductions cannot be isomorphic if $E/F$ is ramified:

\begin{proposition}
\label{supercuspidal vs special: ramified case}
    Let $\rho_1, \rho_2:W_F \rightarrow \GL_2(\overline{\Q}_\ell)$ be a supercuspidal representation and a special representation, respectively, such that $\rho_1(W_F), \rho_2(W_F) \subset \GL_2(\calO_K)$. Then $\Ind_{W_E}^{W_F}\xi$ and $\chi\omega\oplus\chi$ are the Weil--Deligne representations associated with $\rho_1$ and $\rho_2$, respectively, where $\xi$ is a $1$-dimensional representation of $W_E$ and $\chi$ is a $1$-dimensional representation of $W_F$. Using Iwasawa decomposition and \Cref{special representation matrix form}, we can assume without loss of generality that there exist upper triangular matrices 
    \begin{equation*}
        U_1 = \begin{pmatrix}
            a_1 & b_1 \\
            0 & d_1
        \end{pmatrix}, U_2 = \begin{pmatrix}
            a_2 & b_2 \\
            0 & d_2
        \end{pmatrix}
    \end{equation*}
    in $\GL_2(K)$ such that
    \begin{align*}
        \rho_1(\Phi^k\tau) &= U_1\Ind_{W_E}^{W_F}\xi(\Phi^k\tau) U_1^{-1} \\
        \rho_2(\Phi^k\tau) &= \begin{pmatrix}
            \chi\omega(\Phi^k\tau) & \frac{a_2}{d_2}t_\ell(\tau)\chi\omega(\Phi^k\tau) + \frac{b_2}{d_2}(\chi(\Phi^k\tau) - \chi\omega(\Phi^k\tau)) \\
            0 & \chi(\Phi^k\tau)
        \end{pmatrix}
    \end{align*}
    for all $\tau \in I_F, k \in \Z$.

    Suppose that $a_2/d_2 \in \calO_K^\times$. If $E/F$ is ramified, then $\overline{\rho}_1$ and $\overline{\rho}_2$ are not isomorphic.
\end{proposition}

\begin{proof}
    Suppose that $\overline{\rho}_1 \cong \overline{\rho}_2$ and that $E/F$ is ramified. Then $\Tr(\overline{\rho}_1(g)) = \Tr(\overline{\rho}_2(g))$ for all $g \in W_F$. Pick $g \not\in W_E$. Note that 
    \begin{equation*}
        \Tr(\rho_2(g)) = \chi\omega(g)+\chi(g), \Tr(\rho_1(g))=0,
    \end{equation*}
    so
    \begin{equation*}
        \overline{\chi\omega}(g)+\overline{\chi}(g) \equiv 0 \pmod{\lambda},
    \end{equation*}
    or equivalently,
    \begin{equation*}
        \omega(g) \equiv -1 \pmod{\lambda}
    \end{equation*}
    for all $g\not\in W_E$. Pick $g$ in $I_F \setminus W_E$ so that $\omega(g)=1$. (Note that this can be done since $E/F$ is ramified.) Then we get $1 \equiv -1 \pmod{\lambda}$, which is a contradiction. Hence, $\overline{\rho}_1$ and $\overline{\rho}_2$ are not isomorphic.
\end{proof}

Next, we consider the case where $E/F$ is unramified, which is more complicated as $\overline{\rho}_1$ and $\overline{\rho}_2$ can indeed be isomorphic in this case. We start with a few fairly simple lemmas:

\begin{lemma}
\label{lem: omega factors through unramified ext}
    Suppose that $E/F$ is unramified. Then $\ker\overline{\omega} = W_E$ if and only if $q \equiv -1 \pmod{\ell}$.
\end{lemma}

\begin{proof}
    Suppose that $\ker\overline{\omega} = W_E$. Then $\overline{\omega}$ has order exactly $2$ as $W_E$ has index $2$ in $W_F$. Since $\omega(W_F) = q^\Z$, this gives us $q^2 \equiv 1 \pmod{\ell}$, or equivalently $q \equiv \pm 1 \pmod{\ell}$. But $\overline{\omega}$ is nontrivial, so $q \not\equiv 1 \pmod{\ell}$ by \Cref{lem: omega trivial reduction criterion}. We conclude that $q \equiv -1 \pmod{\lambda}$.

    Conversely, suppose that $q \equiv -1 \pmod{\ell}$. We have $\omega(\Phi^k\tau) = q^{-k}$ for all $\tau \in I_F, k \in \Z$, so $\overline{\omega}(\Phi^k\tau) = \overline{1}$ if and only if $k$ is even. Since $E/F$ is unramified, this is equivalent to $\Phi^k\tau \in W_E$. Thus, $\ker\overline{\omega} = W_E$.
\end{proof}

\begin{lemma}
\label{lem: special vs supercuspidal 4 congruences}
    With the same setup as in \Cref{supercuspidal vs special: ramified case}, suppose that $\overline{\rho}_1$ is indecomposable. Then $\overline{\chi}|_{W_E} = \overline{\xi}, q \equiv -1 \pmod{\ell}$, and $\chi(\sigma) \equiv -\frac{a_1}{b_1} \pmod{\lambda}$ if and only if the congruences
    \begin{align}
        \chi\omega(h) &\equiv \xi(h) \pmod{\lambda} \label{4 congruences lemma congruence 1} \\
        \chi(h) &\equiv \xi^\sigma(h) \pmod{\lambda} \label{4 congruences lemma congruence 2} \\
        \chi\omega(h\sigma) &\equiv \frac{b_1}{a_1}\xi(\sigma h \sigma) \pmod{\lambda} \label{4 congruences lemma congruence 3}\\
        \chi(h\sigma) &\equiv -\frac{b_1}{a_1}\xi(\sigma h \sigma) \pmod{\lambda} \label{4 congruences lemma congruence 4}
    \end{align}
    hold for all $h \in W_E$.
\end{lemma}

\begin{proof}
    For the forward implication, note that $\overline{\xi} = \overline{\xi}^\sigma$ by \Cref{thm: decomposability criterion for nonexceptional supercuspidal representations} since $\overline{\rho}_1$ is indecomposable. Therefore, $\overline{\chi}|_{W_E} = \overline{\xi}$ implies \eqref{4 congruences lemma congruence 2}.

    We have
    \begin{equation}
    \label{4 congruences lemma factored omega}
        \overline{\omega}(g) = \begin{cases}
            1, & g \in W_E \\
            -1, & g \not\in W_E
        \end{cases}
    \end{equation}
    by \Cref{lem: omega factors through unramified ext} since $q \equiv -1 \pmod{\ell}$, so we also get \eqref{4 congruences lemma congruence 1}.

    Observe that
    \begin{equation*}
        -\frac{b_1}{a_1}\xi(\sigma h \sigma) \equiv -\chi(\sigma^{-1})\xi(\sigma h \sigma) \equiv -\chi(\sigma^{-1})\chi(\sigma h \sigma) \equiv -\chi(h\sigma) \pmod{\lambda}.
    \end{equation*}
    Lastly, \eqref{4 congruences lemma congruence 3} follows from \eqref{4 congruences lemma congruence 4} and \eqref{4 congruences lemma factored omega}.

    The converse implication follows from similar arguments.
\end{proof}

\begin{lemma}
\label{lem: tame character crossed homomorphism}
    Fix a Frobenius element $\Phi$ and extend $t_\ell$ to $W_F$ by setting $t_\ell(\Phi^k\tau) := t_\ell(\tau) $ for all $\tau \in I_F, k \in \Z$. Then
    $t_\ell(g_1g_2) = \omega(g_2^{-1})t_\ell(g_1) + t_\ell(g_2)$ for all $g_1, g_2 
    \in W_F$.
\end{lemma}

\begin{proof}
    Put $g_1 = \Phi^{k_1}\tau_{1}$ and $g_2 = \Phi^{k_2}\tau_{2}$, where $k_{1}, k_{2} \in \Z$ and $\tau_{1}, \tau_{2} \in I_F$. We can write
    \begin{equation*}
        g_1g_2 = \Phi^{k_1+k_2}(\Phi^{-k_2}\tau_1\Phi^{k_2})\tau_2
    \end{equation*}
    with $(\Phi^{-k_2}\tau_1\Phi^{k_2})\tau_2 \in I_F$, so
    \begin{equation*}
        t_\ell(g_1g_2) = t_\ell((\Phi^{-k_2}\tau_1\Phi^{k_2})\tau_2) = t_\ell(\Phi^{-k_2}\tau_1\Phi^{k_2}) + t_\ell(\tau_2)
    \end{equation*}
    since $\Phi^{-k_2}\tau_1\Phi^{k_2}$ and $\tau_2$ are both in $I_F$. By \cite[32.5.1]{Bushnell-Henniart}, $t_\ell$ satisfies
    \begin{equation*}
        t_\ell(g\tau g^{-1}) = \omega(g)t_\ell(\tau)
    \end{equation*}
    for all $g \in W_F, \tau \in I_F$, so we have
    \begin{equation*}
        t_\ell(\Phi^{-k_2}\tau_1\Phi^{k_2}) = \omega(\Phi^{-k_2})t_\ell(\tau_1).
    \end{equation*}
    Note that $t_\ell(\tau_2) = t_\ell(g_2), t_\ell(\tau_1) = t_\ell(g_1)$, and $\omega(\Phi^{-k_2}) = \omega(g_2)^{-1}$. Hence, we obtain the desired identity.
\end{proof}

\begin{proposition}
\label{thm: special vs supercuspidal comparison unramified case}
    With the same setup as in \Cref{supercuspidal vs special: ramified case}, suppose now that $E/F$ is unramified. Then $\overline{\rho}_1 \cong \overline{\rho}_2$ if and only if $\overline{\rho}_2$ is indecomposable, $q \equiv -1 \pmod{\ell}$, $\overline{\chi}|_{W_E} = \overline{\xi}, \chi(\sigma) \equiv -\frac{a_1}{b_1} \pmod{\lambda}$, and there exists $x \in \calO_K^\times$ such that
    \begin{equation}
        \frac{a_2x}{d_2}t_\ell(h)\chi(h) \equiv \frac{b_1}{d_1}(\xi^\sigma(h) - \xi(h)) \pmod{\lambda}
    \end{equation}
    for all $h \in W_E$, where it is understood that $t_\ell(h)$ denotes the image of the inertial part of $h$ under $t_\ell$ after fixing a Frobenius element $\Phi$.
\end{proposition}

\begin{proof}
    Since $a_2/d_2 \in \calO_K^\times$ by assumption, $\overline{\rho}_2$ is indecomposable by \Cref{thm:decomposability criterion for reductions of special representations}, so we assume that $\overline{\rho}_1$ is also indecomposable. By \Cref{thm: decomposability criterion for nonexceptional supercuspidal representations}, this implies that $\overline{\xi} = \overline{\xi}^\sigma$.

    We can write $\overline{\rho}_1$ and $\overline{\rho}_2$ as
    \begin{equation*}
        \overline{\rho_1}(g) = \begin{pmatrix}
            \overline{\theta}_1(g) & \overline{\psi}_1(g) \\
            0 & \overline{\theta}_2(g)
        \end{pmatrix} , \overline{\rho}_2(g) = \begin{pmatrix}
            \overline{\chi\omega}(g) & \overline{\psi}_2(g) \\
            0 & \overline{\chi}(g)
        \end{pmatrix},
    \end{equation*}
    where
    \begin{alignat*}{2}
        \theta_1(g) &= \begin{cases} \xi(h), \\ \frac{b_1}{a_1}\xi(\sigma h \sigma), \end{cases} 
        &\quad& \begin{aligned} & g=h \in W_E \\ & g = h \sigma, h \in W_E \end{aligned} \\[1.5ex]
        \theta_2(g) &= \begin{cases} \xi^\sigma(h), \\ -\frac{b_1}{a_1}\xi(\sigma h \sigma), \end{cases} 
        &\quad& \begin{aligned} & g=h \in W_E \\ & g = h \sigma, h \in W_E \end{aligned} \\[1.5ex]
        \psi_1(g) &= \begin{cases} \frac{b_1}{d_1}\left(\xi^\sigma(h) - \xi(h)\right), \\ \frac{a_1}{d_1}\xi(h) - \frac{b_1^2}{a_1d_1}\xi(\sigma h \sigma), \end{cases}
        &\quad& \begin{aligned} & g=h \in W_E \\ & g=h\sigma, h \in W_E \end{aligned}
    \end{alignat*}
    and
    \begin{equation*}
       \psi_2(g) = \frac{a_2}{d_2}t_\ell(g)\chi\omega(g) + \frac{b_2}{d_2}(\chi(g) - \chi\omega(g))
    \end{equation*}
    for all $g \in W_F$.

    Applying \Cref{thm: isomorphism of indecomposable representations} to $\overline{\rho}_1$ and $\overline{\rho}_2$, we see that $\overline{\rho}_1 \cong \overline{\rho}_2$ if and only if $\overline{\theta_1} = \overline{\chi\omega}, \overline{\theta}_2 = \overline{\chi}$, and there exist $x \in \calO_K^\times, y \in \calO_K$ such that
    \begin{equation}
    \label{special vs unramified supercuspidal: congruence 1}
        x\psi_2 + y(\chi-\chi\omega) \equiv \psi_1 \pmod{\lambda}.
    \end{equation}
    Writing out the equations $\overline{\theta_1} = \overline{\chi\omega}$ and $\overline{\theta}_2 = \overline{\chi}$ explicitly, we get
    \begin{align*}
        \chi\omega(h) &\equiv \xi(h) \pmod{\lambda} \\
        \chi\omega(h\sigma) &\equiv \frac{b_1}{a_1}\xi(\sigma h \sigma) \pmod{\lambda} \\
        \chi(h) &\equiv \xi^\sigma(h) \pmod{\lambda} \\
        \chi(h\sigma) &\equiv -\frac{b_1}{a_1}\xi(\sigma h \sigma) \pmod{\lambda}
    \end{align*}
    for all $h \in W_E$. By \Cref{lem: special vs supercuspidal 4 congruences}, these congruences are equivalent to $\overline{\chi}|_{W_E} = \overline{\xi}, q \equiv -1 \pmod{\ell}$, and $\chi(\sigma) = -\frac{a_1}{b_1} \pmod{\lambda}$. Therefore, $\overline{\rho}_1 \cong \overline{\rho}_2$ if and only if $\overline{\chi}|_{W_E} = \overline{\xi}, q \equiv -1 \pmod{\ell}, \chi(\sigma) = -\frac{a_1}{b_1} \pmod{\lambda}$, and there exist $x \in \calO_K^\times, y \in \calO_K$ such that \eqref{special vs unramified supercuspidal: congruence 1} holds.

    For $g=h \in W_E$, \eqref{special vs unramified supercuspidal: congruence 1} becomes
    \begin{equation}
    \label{special vs unramified supercuspidal: congruence 2}
        x\left[\frac{a_2}{d_2}t_\ell(h)\chi\omega(h) + \frac{b_2}{d_2}(\chi(h) - \chi\omega(h)) \right] + y(\chi(h) - \chi\omega(h)) \equiv \frac{b_1}{d_1}(\xi^\sigma(h) - \xi(h)) \pmod{\lambda}.
    \end{equation}
    Likewise, \eqref{special vs unramified supercuspidal: congruence 1} can be written explicitly as
    \begin{equation}
    \label{special vs unramified supercuspidal: congruence 3}
        x\left(\frac{a_2}{d_2}t_\ell(h\sigma) \chi\omega(h\sigma) + \frac{b_2}{d_2}(\chi(h\sigma) - \chi\omega(h\sigma) \right)+y(\chi(h\sigma)-\chi\omega(h\sigma)) \equiv \frac{a_1}{d_1}\xi(h) - \frac{b_1^2}{a_1d_1}\xi(\sigma h \sigma) \pmod{\lambda}
    \end{equation}
    for all $g=h\sigma$ with $h \in W_E$.
    
    We assume that \eqref{special vs unramified supercuspidal: congruence 2} implies \eqref{special vs unramified supercuspidal: congruence 3} under the assumptions $\overline{\chi}|_{W_E} = \overline{\xi}, q \equiv -1 \pmod{\ell}$, and $\chi(\sigma) = -\frac{a_1}{b_1} \pmod{\lambda}$, which we prove in \Cref{lem: special vs unramified supercuspidal eqn 1 implies eqn 2}. Therefore, $\overline{\rho}_1 \cong \overline{\rho}_2$ if and only if $\overline{\chi}|_{W_E} = \overline{\xi}, q \equiv -1 \pmod{\ell}, \chi(\sigma) = -\frac{a_1}{b_1} \pmod{\lambda}$, and there exist $x \in \calO_K^\times, y \in \calO_K$ such that \eqref{special vs unramified supercuspidal: congruence 2} holds.

    By \Cref{lem: omega factors through unramified ext}, $q \equiv -1 \pmod{\ell}$ is equivalent to $\overline{\omega}$ factoring through $W_E$, so \eqref{special vs unramified supercuspidal: congruence 2} simplifies to
    \begin{equation}
    \label{special vs unramified supercuspidal: congruence 4}
        x\left[\frac{a_2}{d_2}t_\ell(h)\chi(h) + \frac{b_2}{d_2}(\chi(h) - \chi\omega(h)) \right] \equiv \frac{b_1}{d_1}(\xi^\sigma(h) - \xi(h)) \pmod{\lambda}.
    \end{equation}

    Note that $\psi_2(g) \in \calO_K$ for all $g \in W_F$. Plugging in $g = \Phi$, we find that
    \begin{equation*}
        \frac{b_2}{d_2}\chi(\Phi)(1-q^{-1}) \in \calO_K.
    \end{equation*}
    Since $\chi(\Phi), 1-q^{-1} \in \calO_K^\times$, it follows that $b_2/d_2 \in \calO_K$. Thus, the $\frac{b_2}{d_2}(\chi(h) - \chi\omega(h))$ term in \eqref{special vs unramified supercuspidal: congruence 4} also vanishes, leaving us with
    \begin{equation*}
        \frac{a_2x}{d_2}t_\ell(h)\chi(h) \equiv \frac{b_1}{d_1}(\xi^\sigma(h) - \xi(h)) \pmod{\lambda}.
    \end{equation*}
    This concludes the proof.
\end{proof}

\begin{lemma}
\label{lem: special vs unramified supercuspidal eqn 1 implies eqn 2}
    With the same setup as in \Cref{supercuspidal vs special: ramified case}, if $\overline{\chi}|_{W_E} = \overline{\xi}, q \equiv -1 \pmod{\ell}, \chi(\sigma) = -\frac{a_1}{b_1} \pmod{\lambda}$, and there exists $x \in \calO_K^\times$ such that 
    \begin{equation}
    \label{special vs unramified supercuspidal eqn 1 implies eqn 2: congruence 1}
        \frac{a_2x}{d_2}t_\ell(h)\chi(h) \equiv \frac{b_1}{d_1}(\xi^\sigma(h) - \xi(h)) \pmod{\lambda},
    \end{equation}
    for all $h \in W_E$, then there exists $y \in \calO_K$ such that
    \begin{equation}
    \label{special vs unramified supercuspidal eqn 1 implies eqn 2: congruence 2}
        x\left(\frac{a_2}{d_2}t_\ell(h\sigma) \chi\omega(h\sigma) + \frac{b_2}{d_2}(\chi(h\sigma) - \chi\omega(h\sigma)) \right)+y(\chi(h\sigma)-\chi\omega(h\sigma)) \equiv \frac{a_1}{d_1}\xi(h) - \frac{b_1^2}{a_1d_1}\xi(\sigma h \sigma) \pmod{\lambda}
    \end{equation}
    for all $h \in W_E$.
\end{lemma}

\begin{proof}
    Plugging $h=1$ into \eqref{special vs unramified supercuspidal eqn 1 implies eqn 2: congruence 2}, we get
    \begin{equation}
    \label{special vs unramified supercuspidal eqn 1 implies eqn 2: congruence 3}
        x\left(\frac{a_2}{d_2}t_\ell(\sigma) \chi\omega(\sigma) + \frac{b_2}{d_2}(\chi(\sigma) - \chi\omega(\sigma) \right)+y(\chi(\sigma)-\chi\omega(\sigma)) \equiv \frac{a_1}{d_1} - \frac{b_1^2}{a_1d_1}\xi(\sigma^2) \pmod{\lambda}.
    \end{equation}
    Since $E/F$ is unramified, we can pick $\sigma = \Phi$ so that $t_\ell(\Phi) = 0$ and $\omega(\Phi) = q^{-1} \equiv -1 \pmod{\lambda}$. Then \eqref{special vs unramified supercuspidal eqn 1 implies eqn 2: congruence 3} becomes
    \begin{equation}
    \label{special vs unramified supercuspidal eqn 1 implies eqn 2: congruence 4}
        \frac{2b_2x}{d_2}\chi(\Phi)+2y\chi(\Phi) \equiv \frac{a_1}{d_1} - \frac{b_1^2}{a_1d_1}\xi(\Phi^2) \pmod{\lambda}.
    \end{equation}
    Observe that \eqref{special vs unramified supercuspidal eqn 1 implies eqn 2: congruence 4} can be solved for $y \in \calO_K$. We claim that for this choice of $y$, \eqref{special vs unramified supercuspidal eqn 1 implies eqn 2: congruence 2}  holds for all $h \in W_E$.

    By \Cref{lem: tame character crossed homomorphism}, we have
    \begin{equation*}
        t_\ell(h\Phi) = t_\ell(h\Phi) = \omega(\Phi^{-1})t_\ell(h) + t_\ell(\Phi) = qt_\ell(h) \equiv -t_\ell(h) \pmod{\lambda},
    \end{equation*}
    where the last congruence follows from $q \equiv -1 \pmod{\ell}$. Therefore, \eqref{special vs unramified supercuspidal eqn 1 implies eqn 2: congruence 2} simplifies to
    \begin{equation}
    \label{special vs unramified supercuspidal eqn 1 implies eqn 2: congruence 5}
        x\left(-\frac{a_2}{d_2}t_\ell(h) \chi\omega(h\Phi) + \frac{b_2}{d_2}(\chi(h\Phi) - \chi\omega(h\Phi)) \right)+y(\chi(h\Phi)-\chi\omega(h\Phi)) \equiv \frac{a_1}{d_1}\xi(h) - \frac{b_1^2}{a_1d_1}\xi(\Phi h \Phi) \pmod{\lambda}.
    \end{equation}
    As we have seen in the proof of \Cref{thm: special vs supercuspidal comparison unramified case}, the integrality of $\Psi_1(\Phi)$ implies that $b_2/d_2 \in \calO_K$, so we can further simplify \eqref{special vs unramified supercuspidal eqn 1 implies eqn 2: congruence 5} to obtain
    \begin{equation}
    \label{special vs unramified supercuspidal eqn 1 implies eqn 2: congruence 6}
        x\chi(h\Phi)\left( \frac{2b_2}{d_2} + \frac{a_2}{d_2}t_\ell(h)\right) + 2y\chi(h\Phi) \equiv \frac{a_1}{d_1}\xi(h) - \frac{b_1^2}{a_1d_1}\xi(\Phi h \Phi) \pmod{\lambda}.
    \end{equation}
    Isolating $2y\chi(\Phi)$ in \eqref{special vs unramified supercuspidal eqn 1 implies eqn 2: congruence 4} and substituting it into \eqref{special vs unramified supercuspidal eqn 1 implies eqn 2: congruence 6}, we get
    \begin{align}
        x\chi(h\Phi)\left( \frac{2b_2}{d_2} + \frac{a_2}{d_2}t_\ell(h)\right) + \chi(h)\left( \frac{a_1}{d_1} - \frac{b_1^2}{a_1d_1}\xi(\Phi^2) - \frac{2b_2x}{d_2}\chi(\Phi)\right) &\equiv \frac{a_1}{d_1}\xi(h) - \frac{b_1^2}{a_1d_1}\xi(\Phi h \Phi) \pmod{\lambda} \\
        \frac{a_2x}{d_2}\chi(h\Phi)t_\ell(h) + \chi(h) \left(\frac{a_1}{d_1} - \frac{b_1^2}{a_1d_1}\xi(\Phi^2)\right) &\equiv \frac{a_1}{d_1}\xi(h) - \frac{b_1^2}{a_1d_1}\xi(\Phi h \Phi) \pmod{\lambda} \\
        \frac{a_2x}{d_2}\chi(h\Phi)t_\ell(h) + \frac{a_1}{d_1}\xi(h) - \frac{b_1^2}{a_1d_1}\chi(h)\xi(\Phi^2) &\equiv \frac{a_1}{d_1}\xi(h) - \frac{b_1^2}{a_1d_1}\xi(\Phi h \Phi) \pmod{\lambda} \\
        \frac{a_2x}{d_2}\chi(h\Phi) t_\ell(h) - \frac{b_1^2}{a_1d_1}\chi(h)\xi(\Phi^2) &\equiv  -\frac{b_1^2}{a_1d_1}\xi(\Phi h \Phi) \pmod{\lambda} \\
        \frac{a_2x}{d_2}\chi(h\Phi) t_\ell(h) - \frac{b_1^2}{a_1d_1}\chi(h)\xi(\Phi^2) &\equiv  -\frac{b_1^2}{a_1d_1}\xi^\Phi(h)\xi(\Phi^2) \pmod{\lambda} \\
        \frac{a_2x}{d_2}\chi(h\Phi) t_\ell(h) &\equiv -\frac{b_1^2}{a_1d_1}\xi(\Phi^2) (\xi^\Phi(h) - \xi(h))\pmod{\lambda} \\
        \frac{a_2x}{d_2}\chi(h\Phi) t_\ell(h) &\equiv -\frac{b_1^2}{a_1d_1}\xi(\Phi^2) (\xi^\Phi(h) - \xi(h))\pmod{\lambda} \label{special vs unramified supercuspidal eqn 1 implies eqn 2: congruence 7}
    \end{align}
    It follows directly from the integrality of $\Psi_1(\Phi)$ that
    \begin{equation}
        \left( \frac{a_1}{b_1}\right)^2 \equiv \xi(\Phi^2) \pmod{\lambda}.
    \end{equation}
    This, combined with the congruence $\chi(\sigma) \equiv -\frac{a_1}{b_1} \pmod{\lambda}$, turns \eqref{special vs unramified supercuspidal eqn 1 implies eqn 2: congruence 7} into
    \begin{equation}
        \frac{a_2x}{d_2}\chi(h) t_\ell(h) \equiv \frac{b_1}{d_1}(\xi^\Phi(h) - \xi(h))\pmod{\lambda},
    \end{equation}
    agreeing with \eqref{special vs unramified supercuspidal eqn 1 implies eqn 2: congruence 1}.
\end{proof}

Lastly, we will compare the reductions of two special representations. Once again, we will only deal with the case $a_1/d_1, a_2/d_2 \in \calO_K^\times$ as the other cases reduce to the comparison of principal series against other types.

\begin{proposition}
\label{special vs special decomposable case}
    Let $\rho_1, \rho_2:W_F \rightarrow \GL_2(\overline{\Q}_\ell)$ be two special representations such that $\rho_1(W_F), \rho_2(W_F) \subset \GL_2(\calO_K)$. Then $\chi\omega\oplus\chi$ and $\theta\omega\oplus\theta$ are the Weil--Deligne representations associated with $\rho_1$ and $\rho_2$, respectively, where $\chi, \omega$ are $1$-dimensional representations of $W_F$. By \Cref{special representation matrix form}, we can assume without loss of generality that there exist upper triangular matrices
    \begin{equation*}
        U_1 = \begin{pmatrix}
            a_1 & b_1 \\
            0 & d_1
        \end{pmatrix}, U_2 = \begin{pmatrix}
            a_2 & b_2 \\
            0 & d_2
        \end{pmatrix}
    \end{equation*}
    in $\GL_2(K)$ such that
    \begin{align*}
        \rho_1(\Phi^k\tau) &= \begin{pmatrix}
            \chi\omega(\Phi^k\tau) & \frac{a_1}{d_1}t_\ell(\tau)\chi\omega(\Phi^k\tau) + \frac{b_1}{d_1}(\chi(\Phi^k\tau) - \chi\omega(\Phi^k\tau)) \\
            0 & \chi(\Phi^k\tau)
        \end{pmatrix}\\
        \rho_2(\Phi^k\tau) &= \begin{pmatrix}
            \theta\omega(\Phi^k\tau) & \frac{a_2}{d_2}t_\ell(\tau)\theta\omega(\Phi^k\tau) + \frac{b_2}{d_2}(\theta(\Phi^k\tau) - \theta\omega(\Phi^k\tau)) \\
            0 & \theta(\Phi^k\tau)
        \end{pmatrix}
    \end{align*}
    for all $\tau \in I_F, k \in \Z$.

    Suppose that $a_1/d_1, a_2/d_2 \in \calO_K^\times$. Then $\overline{\rho}_1 \cong \overline{\rho}_2$ if and only if $\overline{\chi} = \overline{\theta}$ and
    \begin{equation}
    \label{special vs special: valuation inequality}
        v\left( \frac{a_2b_1-a_1b_2}{a_1d_2} \right) \geq \min\{0, 1-i\},
    \end{equation}
    where $i$ is the maximal nonnegative integer such that
    \begin{equation*}
        q \equiv 1 \pmod{\lambda^i}.
    \end{equation*}
\end{proposition}

\begin{proof}
    By \Cref{thm: isomorphism of indecomposable representations}, $\overline{\rho}_1 \cong \overline{\rho}_2$ if and only if $\overline{\chi} = \overline{\theta}$ and there exist $x \in \calO_K^\times, y \in \calO_K$ such that
    \begin{align}
    \label{special vs special indecomposable 1st cong}
        & x\left( \frac{a_1}{d_1}t_\ell(\tau)\chi\omega(\Phi^k\tau) + \frac{b_1}{d_1}(\chi(\Phi^k\tau)-\chi\omega(\Phi^k\tau))\right) + y\left(\chi(\Phi^k\tau)-\chi\omega(\Phi^k\tau) \right) \\ \notag & \equiv \frac{a_2}{d_2}t_\ell(\tau)\theta\omega(\Phi^k\tau)+\frac{b_2}{d_2}(\theta(\Phi^k\tau)-\theta\omega(\Phi^k\tau)) \pmod{\lambda}
    \end{align}
    for all $\tau \in I_F, k \in \Z$.  We claim that if $\overline{\chi} = \overline{\theta}$ and $v\left( \frac{a_2b_1}{a_1d_2}-\frac{b_2}{d_2}\right) \geq 1-i$, then we can always find such $x$ and $y$.
    
    Plugging $k = 0$ into \eqref{special vs special indecomposable 1st cong} and picking $\tau_0 \in I_F$ such that $t_{\ell}(\tau_0) \in \calO_K^\times$, we obtain
    \begin{equation*}
        x \equiv \frac{a_2d_1}{a_1d_2} \pmod{\lambda}.
    \end{equation*}
    We will show that this $x$ works for all $k \in \Z, \tau \in I_F$.

    Plugging in $\frac{a_2d_1}{a_1d_2}$ for $x$ in \eqref{special vs special indecomposable 1st cong}, we get
    \begin{align}
        & \frac{a_2d_1}{a_1d_2}\left( \frac{a_1}{d_1}t_\ell(\tau)\chi\omega(\Phi^k\tau) + \frac{b_1}{d_1}(\chi(\Phi^k\tau)-\chi\omega(\Phi^k\tau))\right) + y\left(\chi(\Phi^k\tau)-\chi\omega(\Phi^k\tau) \right) 
        \\ \notag & \equiv \frac{a_2}{d_2}t_\ell(\tau)\theta\omega(\Phi^k\tau)+\frac{b_2}{d_2}(\theta(\Phi^k\tau)-\theta\omega(\Phi^k\tau)) \pmod{\lambda} \\
        & \frac{a_2}{d_2}t_\ell(\tau)\chi\omega(\Phi^k\tau)+\frac{a_2b_1}{a_1d_2}(\chi(\Phi^k\tau)-\chi\omega(\Phi^k\tau))+y(\chi(\Phi^k\tau)-\chi\omega(\Phi^k\tau)) \\
        & \notag \equiv \frac{a_2}{d_2}t_\ell(\tau)\theta\omega(\Phi^k\tau)+\frac{b_2}{d_2}(\theta(\Phi^k\tau)-\theta\omega(\Phi^k\tau)) \pmod{\lambda}
    \end{align}
    The first term on both sides cancels out, and we're left with
    \begin{equation}
        \frac{a_2b_1}{a_1d_2}(\chi(\Phi^k\tau)-\chi\omega(\Phi^k\tau))+y(\chi(\Phi^k\tau)-\chi\omega(\Phi^k\tau)) \equiv \frac{b_2}{d_2}(\theta(\Phi^k\tau)-\theta\omega(\Phi^k\tau)) \pmod{\lambda}
    \end{equation}
    Since $\overline{\chi} = \overline{\theta}$ and $a_2/d_2$ is a unit, we may cancel $\chi(\Phi^k\tau)$ from the left-hand side and $\theta(\Phi^k\tau)$ from the right-hand side to obtain
    \begin{equation}
        \left(\frac{a_2b_1}{a_1d_2}+y \right)(1-\omega(\Phi^k\tau)) \equiv \frac{b_2}{d_2}(1-\omega(\Phi^k\tau)) \pmod{\lambda},
    \end{equation}
    which we can further rearrange to
    \begin{equation}
    \label{special vs special indecomposable 2nd cong}
        \left(\frac{a_2b_1}{a_1d_2}-\frac{b_2}{d_2}+y \right)(1-\omega(\Phi^k\tau)) \equiv 0 \pmod{\lambda}.
    \end{equation}
    Observe that $v(1-\omega(\Phi^k\tau)) = v(1-q^{-k}) \geq v(1-q) = i$ and this lower bound is sharp by the maximality of $i$, so \eqref{special vs special indecomposable 2nd cong} has a solution if and only if there exists $y \in \calO_K$ such that
    \begin{equation}
    \label{special vs special indecomposable 3rd cong}
        v\left( \frac{a_2b_1}{a_1d_2}-\frac{b_2}{d_2}+y\right) \geq 1-i.
    \end{equation}
    We claim that this is equivalent to \eqref{special vs special: valuation inequality}.
    
    First, suppose that $i \geq 1$ so that $\min\{0, 1-i\} = 1-i$. We will show that \eqref{special vs special indecomposable 3rd cong} has a solution if and only if
    \begin{equation}
    \label{special vs special indecomposable 4th cong}
        v\left( \frac{a_2b_1}{a_1d_2}-\frac{b_2}{d_2}\right) \geq 1-i.
    \end{equation}

    Suppose for a contradiction that \eqref{special vs special indecomposable 4th cong} does not hold and there exists $y \in \calO_K$ such that \eqref{special vs special indecomposable 3rd cong} holds. But
    \begin{equation*}
        v\left(\frac{a_2b_1}{a_1d_2}-\frac{b_2}{d_2} \right) < 1-i \leq 0 \leq v(y),
    \end{equation*}
    so
    \begin{equation*}
        v\left( \frac{a_2b_1}{a_1d_2}-\frac{b_2}{d_2}+y\right) = v\left( \frac{a_2b_1}{a_1d_2}-\frac{b_2}{d_2}\right) < 1-i,
    \end{equation*}
    a contradiction. Thus, a solution to \eqref{special vs special indecomposable 3rd cong} forces \eqref{special vs special indecomposable 4th cong}. Conversely, if \eqref{special vs special indecomposable 4th cong} holds, then $y = 0$ is a solution to \eqref{special vs special indecomposable 3rd cong}.

    A similar argument shows that if $i = 0$ so that $\min\{0, 1-i\} = 0$, then \eqref{special vs special indecomposable 3rd cong} has a solution if and only if
    \begin{equation*}
        \frac{a_2b_1}{a_1d_2}-\frac{b_2}{d_2} \in \calO_K.
    \end{equation*}
    This concludes the proof.
\end{proof}

\section{Examples}

Let $r$ be an odd prime. Let $K = \Q(\zeta_r + \zeta_r^{-1})$ where $\zeta_r$ is a primitive $r$th root of unity. Let $\Fp_r$ be a prime in $K$ above $r$ and $\Fq_2$ a prime in $K$ above $2$. In \cite[Theorem 4.1]{BCDF-2025}, it is shown that
\begin{equation}
\label{cong}
  \rhobar_{J_r,\Fp_r} \simeq \rhobar_{L,r} \otimes \chi,
\end{equation}
where $J_r = J_r(t)$ is the abelian variety of $\GL_2(K)$-type in loc.\ cit., $L = L(t)$ is the Legendre curve
\begin{equation}
  L(t) : y^2 = x(x-1)(x-t),
\end{equation}
and $\chi$ is a character of order dividing $2$.

Suppose $(a,b,c) \in \Z^3$ satisfies
\[
    a^r + b^r = c^p,
\]
where $p$ is prime, $(a,b,c) = 1$, $a \equiv 0 \pmod{2}$ and $b \equiv 1 \pmod 2$. Then for 
\[
    t = \frac{a^r}{a^r+b^r},
\]
it is shown in \cite[Corollary 5.5]{BCDF-2025} that the inertial type of $\rho_{J_r,\Fp_r} \mid_{I_{\Fq_2}}$ is a principal series if $r \mid \# \F_{\Fq_2}^\times$ and supercuspidal otherwise. On the other hand, $\rho_{L,r} \otimes \chi \mid_{I_{\Fq_2}}$ is special. Furthermore, the image of inertia of both sides of \eqref{cong} has order $r$. Thus, \eqref{cong} is a source of examples of {\it principal-special} and {\it supercuspidal-special} congruences.

In this section we give examples of principal-special and supercuspidal-special congruences inspired by the above situation, but written in the terminology and notation of our paper.

\begin{remark}
  {\tt ChatGPT} was used to facilitate generalization of our initial direct constructions of such examples, which we then independently explained in terms of the criteria established in this paper.
\end{remark}

\subsection{Principal-Special}

Let $K = \Q_\ell(\zeta_\ell + \zeta_\ell^{-1})$, where $\zeta_\ell$ is a primitive $\ell$th root of unity. Suppose we have a principal series $\rho_{ps}$ on $I_F$ that decomposes as $\delta \oplus \delta^{-1}$, where $\delta:I_F \rightarrow \mu_\ell$ has order $\ell$. For $\tau_0 \in I_F$ such that $\delta(\tau_0) = \zeta_\ell$, we can assume without loss of generality, by rational canonical form, that
\begin{equation}
    \rho_{ps}(\tau_0) = \begin{pmatrix}
        0 & -1 \\
        1 & a_\ell
    \end{pmatrix},
\end{equation}
where $a_\ell := \zeta_\ell + \zeta_\ell^{-1}$. Further conjugating by
\begin{equation*}
  \begin{pmatrix}
      1 & 0 \\
      0 & -1 
  \end{pmatrix}    
\end{equation*}
we may assume
\begin{equation}
    \rho_{ps}(\tau_0) = A = \begin{pmatrix}
        0 & 1 \\
        -1 & a_\ell
    \end{pmatrix}.
\end{equation}
By assumption, $a_\ell$ is in $K$, so $A \in \GL_2(\calO_K)$. Since $A$ generates $\rho_{ps}(I_F)$, we have $\rho_{ps}(I_F) \subseteq \GL_2(\calO_K)$. Note that the eigenvalues of $A$ are $\zeta_\ell^{\pm 1}$. A priori, these values need not be in $K$. By base changing to $L := K(\zeta_\ell)$, however, we can assume that they lie in the coefficient field. Then we have
\begin{equation}
    \begin{pmatrix}
        1 & 1 \\
        \zeta_\ell & \zeta_\ell^{-1}
    \end{pmatrix}
    \begin{pmatrix}
        \zeta_\ell & 0 \\
        0 & \zeta_\ell^{-1}
    \end{pmatrix}
    \begin{pmatrix}
        1 & 1 \\
        \zeta_\ell & \zeta_\ell^{-1}
    \end{pmatrix}^{-1} = A
\end{equation}
and the matrix
\begin{equation}
    M := \begin{pmatrix}
        1 & 1 \\
        \zeta_\ell & \zeta_\ell^{-1}
    \end{pmatrix}
\end{equation}
is in $\GL_2(L)$. Using Iwasawa decomposition, we can write $M$ as
\begin{equation}
    M = \begin{pmatrix}
        1 & 0 \\
        \zeta_\ell & 1
    \end{pmatrix}\begin{pmatrix}
        1 & 1 \\
        0 & \zeta_\ell^{-1} - \zeta_\ell
    \end{pmatrix}.
\end{equation}
Let
\begin{equation}
    U = \begin{pmatrix}
        1 & 1 \\
        0 & \zeta_\ell^{-1} - \zeta_\ell
    \end{pmatrix}.
\end{equation}
Comparing this with \Cref{thm:decomposability criterion for reductions of principal series}, we have $a_1 = 1, b_1 = 1, d_1 = \zeta_\ell^{-1} - \zeta_\ell$. To determine the valuation of $d_1$ in $L$, we write it as
\begin{equation}
\zeta_\ell^{-1} - \zeta_{\ell} = \zeta_\ell^{-1}(1-\zeta_\ell)(1+\zeta_\ell).
\end{equation}
We look at each factor separately. Note that $\zeta_\ell^{-1} \in \calO_L^\times$. Further, $1-\zeta_\ell$ is a uniformizer for L, so $v(1-\zeta_\ell) = 1$. Finally, since $\zeta_\ell \equiv 1 \pmod{\lambda}$, we have $\zeta_\ell + 1 \equiv 2 \pmod{\lambda}$, so $\zeta_\ell + 1$ is also in  $\calO_K^\times$ as $\ell$ is odd. Hence, $v(\zeta_\ell^{-1}-\zeta_\ell) = 1$, so $v(b_1/d_1) = -1.$

Next, we describe the characters $\delta$ and $\delta^{-1}$. The image of $\delta$ is a finite, cyclic group of order $\ell \neq p$. Since $P_F$ is a pro-$p$ group, we must have $P_F \subseteq \ker\delta$, meaning $\delta$ factors through $I_F/P_F \cong \prod_{m \neq p} \Z_m$. For the image to be of order $\ell$, it must project onto the $\Z_\ell$ factor, which gives us $t_\ell$. We then reduce $t_\ell$ mod $\ell\Z_\ell$ to get a representation over $\Z/\ell \Z$, which gives us a representation mapping to $\mu_\ell(\calO_L)$, the $\ell$th roots of unity in $\calO_L$. Therefore, for any $\tau \in I_F$, we have
\begin{equation}
\label{delta-formula}
  \delta(\tau) = \zeta_\ell^{t_\ell(\tau)}.
\end{equation}
where it is understood that $t_\ell(\tau) \in \Z_\ell$ is reduced modulo $\ell$ in the above formula.

We now confirm that the maximal nonnegative integer $i$ such that
\begin{equation}
    \delta(\tau) \equiv \delta^{-1}(\tau) \pmod{\lambda^i}
\end{equation}
for all $\tau \in I_F$ is $i = 1$, which implies by \Cref{thm:decomposability criterion for reductions of principal series} on the decomposability of reductions of principal series that $\overline{\rho}_{ps}$ is indecomposable.

By \eqref{delta-formula} we have
\begin{equation}
    \delta(\tau) - \delta^{-1}(\tau) \equiv \zeta_\ell^{t_\ell(\tau)} - \zeta_\ell^{-t_\ell(\tau)} \pmod{\lambda^i}.
\end{equation}
Pick $\tau_1 \in I_F$ such that $t_\ell(\tau_1) \equiv 1 \pmod{\ell}$. Then
\begin{equation}
    \zeta_\ell^{t_\ell(\tau_1)} - \zeta_\ell^{-t_\ell(\tau_1)} \equiv \zeta_\ell^{-1} (\zeta_\ell-1)(\zeta_\ell +1) \pmod{\lambda^ i}.
\end{equation}
Once again, the first factor and the third factors are units, and the second factor is a uniformizer, so $i = 1$.

Next, we construct a special representation $\rho_{sp}$. Let $\chi$ be an unramified $1$-dimensional representation and take $a_2/d_2 = 1$ so that
\begin{equation}
\label{principal vs special example matrix 1}
    \rho_{sp}(\tau) = \begin{pmatrix}
        1 & t_\ell(\tau) \\
        0 & 1
    \end{pmatrix}
\end{equation}
for all $\tau \in I_F$. As $a_2/d_2 \in \calO_K^\times$, $\overline{\rho}_{sp}$ is indecomposable by \Cref{thm:decomposability criterion for reductions of special representations}.

By \Cref{thm: principal vs special}, $\overline{\rho}_{ps} \cong \overline{\rho}_{sp}$ if and only if $\overline{\delta} = \overline{\delta}^{-1} = \overline{1}$ and there exists $x \in \calO_L^\times$ such that
\begin{equation}
\label{special vs principal example isomorphism congruence}
    xt_\ell(\tau) \equiv \frac{1}{\zeta_\ell^{-1} - \zeta_\ell}(\delta^{-1}(\tau) - \delta(\tau)) \pmod{\lambda}
\end{equation}
for all $\tau \in I_F$. Putting $\delta(\tau) = \zeta_\ell^{t_\ell(\tau)}$ once again and plugging in $\tau = \tau_1$, \eqref{special vs principal example isomorphism congruence} becomes
\begin{equation}
    x \equiv \frac{1}{\zeta_\ell^{-1} - \zeta_\ell}(\zeta_\ell^{-1} - \zeta_\ell) \equiv 1 \pmod{\lambda}.
\end{equation}
We will show that this choice of $x$ works for all $\tau \in I_F$, i.e.,
\begin{equation}
\label{principal vs special example final congruence}
    t_\ell(\tau) \equiv \frac{\delta^{-1}(\tau) - \delta(\tau) }{\zeta_\ell^{-1} - \zeta_\ell } \pmod{\lambda}
\end{equation}
for all $\tau \in I_F$.

We have
\begin{align}
    \frac{\delta^{-1}(\tau) - \delta(\tau) }{\zeta_\ell^{-1} - \zeta_\ell } \equiv \frac{\zeta_\ell^{-t_\ell(\tau)} - \zeta_\ell^{t_\ell(\tau)} }{\zeta_\ell^{-1} - \zeta_\ell } &\equiv \frac{\zeta_\ell^{-t_\ell(\tau)}}{\zeta_\ell^{-1} } \frac{1 - \zeta_\ell^{2t_\ell(\tau)} }{1 - \zeta_\ell^2 } \pmod{\lambda} \\
    &\equiv \frac{\zeta_\ell^{-t_\ell(\tau)}}{\zeta_\ell^{-1} }\frac{(1-\zeta_\ell^{t_\ell(\tau)})(1+ \zeta_\ell^{t_\ell(\tau)})  }{(1-\zeta_\ell)(1+\zeta_\ell)} \pmod{\lambda}.
\end{align}
Observe that the terms $\frac{\zeta_\ell^{-t_\ell(\tau)}}{\zeta_\ell^{-1} }, \frac{1-\zeta_\ell^{t_\ell(\tau)}}{1-\zeta_\ell}$, and $\frac{1+ \zeta_\ell^{t_\ell(\tau)}}{1+\zeta_\ell}$ are each in $\calO_K$ separately. Since $\zeta_\ell \equiv 1 \pmod{\lambda}$, the first term is simply $1$ mod $\lambda$. By the same argument, the last term is also $1$ mod $\lambda$. Finally, we have
\begin{equation}
    \frac{1-\zeta_\ell^{t_\ell(\tau)}}{1-\zeta_\ell} \equiv \frac{(1-\zeta_\ell)(1 + \zeta_\ell + \dots + \zeta_\ell^{t_\ell(\tau)-1} )}{1-\zeta_\ell} \equiv 1 + \zeta_\ell + \dots + \zeta_\ell^{t_\ell(\tau)-1} \equiv t_\ell(\tau) \pmod{\lambda}.
\end{equation}
Hence, \eqref{principal vs special example final congruence} holds and we have produced a congruence between a principal series and special representation on inertia.



\subsection{Supercuspidal-Special}

Let $E/F$ be the unramified quadratic extension and $\rho_{sc} = \Ind_{W_E}^{W_F}\theta$ be a nonexceptional supercuspidal representation of $W_F$ such that the restriction  $\delta = \theta|_{I_F}$ has order $\ell$. Normalize $\theta$ so that $\theta(\Phi^2) = 1$. Then $\rho_{sc}|_{I_F} = \delta \oplus \delta^\sigma$, and since $E/F$ is unramified, we can take $\sigma$ to be $\Phi$ so that $\rho_{sc}|_{I_F} = \delta \oplus \delta^\Phi$. Note that we have $\delta(\tau) = \zeta_\ell^{t_\ell(\tau)}$ by the same argument that we have used in the previous example, so
\begin{equation*}
    \delta^\Phi(\tau) = \delta(\Phi \tau \Phi^{-1}) = \zeta_\ell^{t_\ell(\Phi \tau \Phi^{-1})} = \zeta_\ell^{q^{-1}t_\ell(\tau)} = \delta(\tau)^{q^{-1}}.
\end{equation*}

Since $(\delta^\Phi)^\Phi = \delta$, we have $\delta^{q^2} = \delta$, so
\begin{equation*}
    q^2 \equiv 1 \pmod{\ell},
\end{equation*}
as $\delta$ has order $\ell$. However, $\delta^\Phi \neq \delta$ as $\rho_{sc}$ is irreducible, so we find
\begin{equation*}
    q \equiv -1 \pmod{\ell}.
\end{equation*}
Thus, $\rho_{sc}|_{I_F}$ decomposes as $\delta \oplus \delta^{-1}$.

Let $\tau_0 \in I_F$ such that $\delta(\tau_0) = \zeta_\ell$. We have seen that, after a change of basis, $\rho_{sc}(\tau_0)$ is of the form
\begin{equation*}
    \rho_{sc}(\tau_0) = \begin{pmatrix}
        \delta(\tau_0) & \frac{b_1}{d_1}(\delta^\Phi(\tau_0) - \delta(\tau_0)) \\
        0 & \delta^\Phi(\tau_0)
    \end{pmatrix} = \begin{pmatrix}
        \zeta_\ell & \frac{b_1}{d_1}(\zeta_\ell^{-1} - \zeta_\ell) \\
        0 & \zeta_\ell^{-1}
    \end{pmatrix}.
\end{equation*}
We pick $b_1, d_1$ so that $b_1/d_1 = (\zeta_\ell^{-1} - \zeta_\ell)^{-1}$. We can write
\begin{equation*}
    \zeta_\ell^{-1}-\zeta_\ell = \zeta_\ell^{-1}(1-\zeta_\ell)(1+\zeta_\ell),
\end{equation*}
and since $\zeta_\ell, 1+\zeta_\ell$ are both in $\calO_K^\times$, we have $d_1/b_1 \in \lambda$.

Next, we look at the image of $\Phi$. Since $\theta(\Phi^2) =1$, we have
\begin{equation*}
    \rho_{sc}(\Phi) = \begin{pmatrix}
        \frac{b_1}{a_1} & \frac{a_1}{d_1} - \frac{b_1^2}{a_1d_1} \\
        \frac{d_1}{a_1} & -\frac{b_1}{a_1}
    \end{pmatrix}.
\end{equation*}
We pick $a_1$ to be $-b_1$ so that $\rho_{sc}(\Phi)$ becomes
\begin{equation*}
    \rho_{sc}(\Phi) = \begin{pmatrix}
        -1 & 0 \\
        -(\zeta_\ell^{-1}-\zeta_\ell) & 1
    \end{pmatrix}.
\end{equation*}
We have
\begin{equation}
\label{special vs supercuspidal example: supercuspidal matrices}
    \overline{\rho}_{sc}(\tau_0) = \begin{pmatrix}
        1 & 1 \\
        0 & 1
    \end{pmatrix},
    \overline{\rho}_{sc}(\Phi) = \begin{pmatrix}
        -1 & 0 \\
        0 & 1
    \end{pmatrix} = \begin{pmatrix}
        q & 0 \\
        0 & 1
    \end{pmatrix}.
\end{equation}
Note that $\overline{\rho}_{sc}$ is reducible by \Cref{thm: decomposability criterion for nonexceptional supercuspidal representations} since $d_1/a_1 \in \lambda$, but it is indecomposable as $\overline{\rho}_{sc}(\tau)$ is not diagonalizable.

Next, we construct a special representation $\rho_{sp}$ with an unramified character $\chi$. Once again, normalize so that $\chi(\Phi) = 1$. As we have seen in \Cref{special representation matrix form}, $\rho_{sp}$ is of the form
\begin{equation*}
    \rho_{sp}(\Phi^k\tau) = \begin{pmatrix}
            \chi\omega(\Phi^k\tau) & \frac{a_2}{d_2}t_\ell(\tau)\chi\omega(\Phi^k\tau) + \frac{b_2}{d_2}(\chi(\Phi^k\tau) - \chi\omega(\Phi^k\tau)) \\
            0 & \chi(\Phi^k\tau)
        \end{pmatrix}.
\end{equation*}
For $k=0, \tau=\tau_0$, this becomes
\begin{equation*}
    \rho_{sp}(\tau) = \begin{pmatrix}
        1 & \frac{a_2}{d_2}t_\ell(\tau) \\
        0 & 1
    \end{pmatrix}.
\end{equation*}
Note that $\delta(\tau_0) = \zeta_\ell^{t_\ell(\tau_0)} = \zeta_\ell$, so $t_\ell(\tau_0) \equiv 1 \pmod{\ell}$. Let $a_2/d_2 = t_\ell(\tau_0)^{-1}$ so that
\begin{equation*}
    \rho_{sp}(\tau_0) = \begin{pmatrix}
        1 & 1 \\
        0 & 1
    \end{pmatrix}.
\end{equation*}
As $a_2/d_2 \in \calO_K^\times$, $\overline{\rho}_{sp}$ is indecomposable by \Cref{thm:decomposability criterion for reductions of special representations}. We also have
\begin{equation*}
    \rho_{sp}(\Phi) = \begin{pmatrix}
        q^{-1} & \frac{b_2}{d_2}(1-q^{-1}) \\
        0 & 1
    \end{pmatrix}.
\end{equation*}
We pick $b_2 = 0$ so that
\begin{equation*}
    \rho_{sp}(\Phi) = \begin{pmatrix}
        q^{-1} & 0 \\
        0 & 1
    \end{pmatrix}.
\end{equation*}
Observe that
\begin{equation*}
    \overline{\rho}_{sp}(\tau_0) = \begin{pmatrix}
        1 & 1 \\
        0 & 1
    \end{pmatrix}, \overline{\rho}_{sp}(\Phi) = \begin{pmatrix}
        q^{-1} & 0 \\
        0 & 1
    \end{pmatrix} = \begin{pmatrix}
        q & 0 \\
        0 & 1
    \end{pmatrix}.
\end{equation*}

Lastly, we check $\overline{\rho}_{sc}$ and $\overline{\rho}_{sp}$ against our isomorphism criterion in \Cref{thm: special vs supercuspidal comparison unramified case}. We have already seen that $q \equiv -1 \pmod{\ell}$. Since $\overline{\delta} = \overline{\chi}|_{I_F} = \overline{1}$ and $\theta(\Phi^2) = \chi(\Phi^2) = 1$, we have $\overline{\theta} = \overline{\chi}|_{W_E}$. Further, $\chi(\Phi) = 1 = -\frac{a_1}{b_1}$. Finally, let $x = 1$. Since $\tau_0$ generates the image of $I_F$, it suffices to look at elements of the form $\Phi^{2k}\tau_0^s$, where $k, s \in \Z$. We have
\begin{align*}
    \frac{ax}{d}t_\ell(\Phi^{2k}\tau_0^s)\chi(\Phi^{2k}\tau_0^s) \equiv t_\ell(\tau_0^s) \equiv st_\ell(\tau_0) &\equiv s \pmod{\lambda} \\
    &\equiv \frac{1}{\zeta_\ell^{-1} - \zeta_\ell} (\zeta_\ell^{-s} - \zeta_\ell^s)\\
    &\equiv \frac{n}{r}(\theta^\Phi(\Phi^{2k}\tau_0^s) - \theta(\Phi^{2k}\tau_0^s)) \pmod{\lambda}.
\end{align*}
Hence, $\overline{\rho}_{sc} \cong \overline{\rho}_{sp}$ by \Cref{thm: special vs supercuspidal comparison unramified case}.

\bibliographystyle{acm}
\bibliography{main}

\end{document}